\documentclass[11pt,reqno]{amsart}
\usepackage[T1]{fontenc}
\usepackage{lmodern}
\usepackage{amsmath,amssymb,mathtools,mathrsfs}
\usepackage{microtype}
\usepackage{needspace}
\usepackage{tikz-cd}
\usepackage{xurl}
\usepackage[colorlinks=true,linkcolor=blue,citecolor=blue,urlcolor=blue]{hyperref}
\numberwithin{equation}{section}
\allowdisplaybreaks
\newtheorem{theorem}{Theorem}[section]
\newtheorem{proposition}[theorem]{Proposition}
\newtheorem{lemma}[theorem]{Lemma}
\newtheorem{corollary}[theorem]{Corollary}
\newtheorem{problem}[theorem]{Problem}
\newtheorem{maintheorem}{Theorem}

\theoremstyle{definition}
\newtheorem{definition}[theorem]{Definition}
\theoremstyle{remark}
\newtheorem{remark}[theorem]{Remark}
\DeclareMathOperator{\Vol}{Vol}

\DeclareMathOperator{\HSC}{HSC}

\title[Bisectional curvature of relative K\"ahler fibrations]
{Non-strict negativity of holomorphic bisectional curvature
for compact relative K\"ahler fibrations}
\author{Xueyuan Wan}
\address{Mathematical Science Research Center, Chongqing University of Technology,
Chongqing 400054, China}
\email{xwan@cqut.edu.cn}
\date{}

\makeatletter
\@namedef{subjclassname@2020}{\textup{2020} Mathematics Subject Classification}
\makeatother
\hypersetup{
 pdftitle={Non-strict negativity of holomorphic bisectional curvature
for compact relative K\"ahler fibrations},
 pdfauthor={Xueyuan Wan and Xu Wang},
 pdfsubject={Complex differential geometry and relative K\"ahler fibrations}}

\begin{document}

\begin{abstract}

Strict negativity of holomorphic bisectional curvature need not pass from the base and fibers of a compact holomorphic fibration to its total space, even when the Kodaira-Spencer map is everywhere injective. We construct a compact relative K\"ahler fibration over a genus-two
curve, with a smooth projective threefold as total space and an everywhere injective
Kodaira--Spencer map, whose induced fiber metrics have strictly
negative holomorphic bisectional curvature, whereas the total space
admits no K\"ahler metric with this curvature property. This gives a negative answer to an open problem posed by To and Yeung. We also construct compact, effectively parametrized Monge-Amp\`ere fibrations from universal quaternionic abelian surfaces over Shimura curves. For products of these families, the generalized Weil-Petersson metric has nonpositive holomorphic bisectional curvature and strictly negative holomorphic sectional curvature, but its mixed bisectional curvatures vanish. Thus strict bisectional negativity can fail both for the existence of a metric on the total space and for the natural metric on the parameter space, despite effective variation at every point.
\end{abstract}
\maketitle

\section*{Introduction}

Negative curvature is an important theme in complex geometry. Alongside the study of its geometric and analytic consequences, a basic problem is to construct K\"ahler metrics satisfying a prescribed negativity condition. Holomorphic sectional curvature and holomorphic bisectional curvature are particularly natural in this setting, although the strict negativity of the latter is a substantially stronger requirement. Holomorphic fibrations provide a useful setting for studying these conditions: the geometry of the total space reflects both the intrinsic geometry of the fibers and the variation of their complex structures over the base.

Throughout this paper, a \emph{holomorphic fibration} is a proper
surjective holomorphic submersion with connected fibers. It is called
\emph{compact} when its total space and base are compact connected
complex manifolds. A \emph{relative K\"ahler form} for a fibration
$p:X\to B$ is a smooth real $d$-closed $(1,1)$-form $\omega$ on $X$
whose restriction
\(
 \omega_b:=\omega|_{X_b},\,X_b=p^{-1}(b),
\)
is positive definite on every fiber. We then call
$p:(X,\omega)\to B$ a \emph{relative K\"ahler fibration}; the form
$\omega$ need not be positive definite on the whole total space.
The closedness of $\omega$ imposes a compatibility condition on
the fiber metrics and is essential in constructing K\"ahler
metrics on $X$. For a compact family over a K\"ahler base
$(B,\omega_B)$, the form $\omega+k\,p^*\omega_B$ is K\"ahler on
$X$ for all sufficiently large $k$.

For compact surfaces admitting a Kodaira fibration, several results
illustrate the role of variation in obtaining strict negativity.
Cheung \cite{Che89} constructed K\"ahler metrics of negative
holomorphic sectional curvature, while Tsai \cite{Tsa89} constructed
Hermitian metrics of negative holomorphic bisectional curvature.
To and Yeung \cite[Theorem~1]{TY11} proved that every Kodaira
surface, in their sense, admits a K\"ahler metric of strictly
negative holomorphic bisectional curvature. Here the
Kodaira--Spencer map is required to be injective at every point of
the base. Their construction associates to each point of the total
space the corresponding fiber marked at that point. The resulting
map to the moduli space of pointed curves lifts locally to
holomorphic immersions into Teichm\"uller space, and the desired
metric is induced by the Weil--Petersson metric.

In \cite[Remark~1]{TY11}, To and Yeung observed that the same
argument applies to effectively parametrized families of compact
curves of genus at least two over compact complex bases of
arbitrary dimension. The total space again admits a K\"ahler
metric of strictly negative holomorphic bisectional curvature.
On the other hand,
Wan \cite[Corollary~1.4]{Wan26} gave another proof by a direct
metric construction. The relative form induced by the fiberwise
hyperbolic metrics yields metrics of the form
$\omega+k\,p^*\omega_B$, $k\gg1$. More generally,
\cite[Theorem~1.2]{Wan26} gives strict bisectional negativity on
the total space when the relative form induces a Griffiths
negative metric on $T_{X/B}$ and the base has strictly negative
holomorphic bisectional curvature. For holomorphic sectional
curvature, the corresponding result requires only strict
negativity of the induced fiber metrics and of the base metric
\cite[Theorem~1.6]{Wan26}. The distinction is that Griffiths
negativity of $T_{X/B}$ tests curvature in every direction on
$X$, including directions transverse to the fibers.

For higher-dimensional fibers, To and Yeung posed the following
problem in \cite[Remark~2(i)]{TY11}. We reproduce their formulation
verbatim:
\begin{problem}[To-Yeung]\label{problem1}
\emph{In general given a non-trivial fibration for which fibers and base
are all equipped with K\"ahler metrics of negative holomorphic
bisectional curvature, one may ask whether the total space of the
fibration admits a K\"ahler metric with a similar curvature property.
Theorem 1 provides an affirmative example to such a problem.}
\end{problem}
The theorem mentioned in the quotation is their theorem for Kodaira
surfaces. We give a negative answer within the class of compact
relative K\"ahler fibrations. In our example, the negatively
curved fiber metrics are the restrictions of a single relative
K\"ahler form, and the Kodaira--Spencer map is everywhere injective.

For a fibration $p:X\to B$, write
\(
 \rho_{p,b}:T_bB\to H^1(X_b,T_{X_b})
\)
for the classical Kodaira--Spencer map. The family is
\emph{effectively parametrized} if $\rho_{p,b}$ is injective for
every $b\in B$. We use \emph{strictly negative holomorphic
bisectional curvature} to mean that
$R(u,\bar u,v,\bar v)<0$ for every pair of nonzero complex
tangent vectors at the same point. Effectivity is necessary for
such a metric to exist on the total space of a fibration with
positive-dimensional compact fibers; see
Lemma~\ref{ce335:lem:negative-vanishing} and \cite{MR436056}. We can give 
a negative answer to Problem \ref{problem1}, and obtain

\begin{maintheorem}\label{ce335:thm:main}
There exist a smooth connected projective threefold $X$, a smooth
projective curve $C$ of genus two, and a compact relative K\"ahler
fibration
\[
 \pi:(X,\omega)\longrightarrow C
\]
whose underlying map is a smooth surjective projective morphism
with connected surface fibers, such that:
\begin{enumerate}
\renewcommand{\labelenumi}{\textup{(\roman{enumi})}}
\item the induced K\"ahler metric $\omega_c:=\omega|_{X_c}$ has
strictly negative holomorphic bisectional curvature for every
$c\in C$, and $C$ admits a hyperbolic K\"ahler metric;
\item the Kodaira--Spencer map of $\pi$ is injective at every point;
\item $X$ admits no K\"ahler metric of strictly negative
holomorphic bisectional curvature.
\end{enumerate}
The relative form $\omega$ can be chosen to be K\"ahler on $X$.
Moreover, $X$ admits a K\"ahler metric of strictly negative
holomorphic sectional curvature.

More precisely, $X$ contains a smooth connected curve $G$ of genus
$14$, with $\deg(\pi|_G)=9$, and holomorphic subbundle inclusions
\[
 \mathcal O_G\hookrightarrow T_{X/C}|_G\hookrightarrow T_X|_G.
\]
In particular, the cotangent bundle $\Omega_X^1$ is not ample.
\end{maintheorem}

Theorem~\ref{ce335:thm:main} shows that the positive result for
families of curves can fail for families of surfaces, even when
the base is still a curve. This happens even though all the fiber
metrics are induced by a single K\"ahler form on $X$ and the
Kodaira--Spencer map is injective at every point. The obstruction
is the trivial holomorphic line subbundle
$\mathcal O_G\hookrightarrow T_X|_G$, which prevents $X$ from
carrying any K\"ahler metric with strictly negative holomorphic
bisectional curvature. The curve $G$ maps onto the base and is
therefore not contained in any fiber. Thus this obstruction does
not contradict the strict negativity of the induced metrics on
the individual fibers. The existence of a
metric of strictly negative holomorphic sectional curvature
follows from \cite[Theorem~1.6]{Wan26}; thus the two curvature
conditions lead to different conclusions on the same total space.

Two successive families of triple covers, each ramified at one
moving point, first give a negatively curved projective threefold
$Y$ with a curve fibration $p:Y\to T$ over a surface and a surface
fibration $f:Y\to C$ over a genus-two curve. After a finite \'etale
change of $T$, we form a cyclic cover of $T$ of degree five and
pull it back to $Y$. On the fibers of $f$, this gives cyclic covers
branched along disjoint negative curves. A local curvature calculation gives negatively
curved K\"ahler metrics as global potential perturbations of one
fixed background form. Proposition~\ref{ce335:prop:relative-negative-form}
then assembles these potentials into a relative K\"ahler form. Viewed through $p$, the same
cover is a ramified change of parameter. A vector in the kernel
of its differential defines a trivial holomorphic tangent line
along a compact curve fiber. Finally, trace and functoriality of
Kodaira--Spencer classes preserve effectivity of the outer surface
family.

Our second theme concerns the metric induced on the base by a
special kind of relative K\"ahler form. A relative K\"ahler
fibration $p:(X,\omega)\to B$ of relative dimension $n$ is called
\emph{Monge--Amp\`ere} if
\(
 \omega^{n+1}=0
\)
on $X$. The relative form then has constant complex rank $n$. It
determines horizontal lifts and Kodaira--Spencer tensors whose
fiberwise $L^2$ pairing defines the generalized Weil--Petersson
metric introduced in \cite[Definition~1.12]{WanWang2023}.
For an effectively parametrized family, this is a positive-definite
K\"ahler metric with nonpositive holomorphic bisectional curvature
and strictly negative holomorphic sectional curvature
\cite[Theorem~2.4 and Corollary~2.6]{WanWang2023}.

The existence of compact Monge--Amp\`ere fibrations with everywhere
injective Kodaira--Spencer map is an interesting problem in its own
right. The Monge--Amp\`ere equation requires the relative form to
be degenerate in the horizontal directions, whereas injectivity
of the Kodaira--Spencer map requires nontrivial complex variation
in every direction on the base. Constructing a family that
satisfies both conditions globally, with compact total space and
base, also makes it possible to study the generalized
Weil--Petersson metric in a compact setting. We obtain such
relative K\"ahler fibrations from universal abelian surfaces over
quaternionic Shimura curves.
\begin{maintheorem}\label{thm:compact-Shimura-MA-example}
There exist a smooth projective curve $C$, a smooth projective
threefold $\mathcal A$, and a smooth projective morphism
\[
 \pi:\mathcal A\longrightarrow C
\]
with connected abelian-surface fibers, together with a smooth real
closed semipositive $(1,1)$-form $\omega_{\mathrm B}$ of constant
complex rank two, such that
\[
 \omega_{\mathrm B}|_{\mathcal A_c}>0,
 \qquad \omega_{\mathrm B}^3=0,
\]
and the Kodaira--Spencer map is injective at every point of $C$.
The family can be taken to be the universal quaternionic abelian
surface over a compact Shimura curve.
\end{maintheorem}

The form $\omega_{\mathrm B}$ is the Betti form associated with the
principal polarization. It makes this family a compact relative
K\"ahler fibration whose induced fiber metrics are flat. Here the
curvature under consideration is that of the generalized
Weil--Petersson metric on the base. The Betti form has the following local expression in flat
symplectic coordinates:
\(
 \omega_{\mathrm B}=2\sum_{\alpha=1}^2 dx_\alpha\wedge dy_\alpha.
\)
The symplectic invariance of this expression gives a global closed
form, and the period-matrix description shows that it has type
$(1,1)$, is positive on the fibers, and has constant complex rank
two. We verify that the Kodaira--Spencer map is injective everywhere
by an explicit nonzero harmonic representative; this is also
consistent with Yuan's formula \cite[Theorem~4.2]{Yuan2024}.
The construction uses the classical quaternionic family and its
polarization; our purpose is to make the compact Monge--Amp\`ere
structure and its effective variation explicit.

These examples allow us to examine whether the nonpositivity of the
generalized Weil--Petersson metric can be strengthened to strict
bisectional negativity. For comparison, the classical
Weil--Petersson metric on Teichm\"uller space has strictly negative
holomorphic bisectional curvature
\cite[Theorem~4.5]{Wolpert1986}; the pointed case is treated in
\cite[Proposition~1]{TY11}. In the Monge--Amp\`ere setting, however,
the curvature estimate allows equality for orthogonal tangent
directions. We show that this equality occurs on compact bases even
when the Kodaira--Spencer map is everywhere injective.

The reason is a product formula. Products of Monge--Amp\`ere
fibrations are again Monge--Amp\`ere, and their generalized
Weil--Petersson metrics split as products, with constant positive
factors given by the fiber volumes. Consequently, tangent directions
coming from distinct factors have zero bisectional curvature.
Applying this construction to two families from
Theorem~\ref{thm:compact-Shimura-MA-example} gives the following result.

\begin{maintheorem}\label{thm:compact-product-counterexample}
There exist a smooth projective sixfold $\mathcal X$, a smooth
projective surface $B$, and an effectively parametrized
Monge--Amp\`ere fibration
\[
 p:(\mathcal X,\omega)\longrightarrow B
\]
with connected abelian-fourfold fibers such that its generalized
Weil--Petersson metric is K\"ahler, has nonpositive holomorphic
bisectional curvature and strictly negative holomorphic sectional
curvature, and has zero bisectional curvature on some pair of nonzero
tangent directions at every point of $B$.
More precisely, one can take $B=C_1\times C_2$ and
$\mathcal X=\mathcal A_1\times\mathcal A_2$, using two families from
Theorem~\ref{thm:compact-Shimura-MA-example}. If
$V_i=\int_{\mathcal A_{i,c_i}}\omega_{\mathrm B,i}^2/2!$, then
\[
 G_{\mathrm{WP}}=
 V_2\operatorname{pr}_1^*G_{\mathrm{WP},1}
 +V_1\operatorname{pr}_2^*G_{\mathrm{WP},2}.
\]
Every bisectional curvature between tangent directions in the two
different factors is zero.
\end{maintheorem}

Every nonzero tangent direction in
Theorem~\ref{thm:compact-product-counterexample} represents a
nontrivial deformation of the fiber. The zero mixed bisectional
curvatures arise because the two factors deform independently,
while the generalized Weil--Petersson metric remains positive
definite and has strictly negative holomorphic sectional curvature.
Thus effective variation and compactness do not imply strict
bisectional negativity for this metric. Together with
Theorem~\ref{ce335:thm:main}, this exhibits two distinct ways in
which strict negativity can fail for compact relative K\"ahler fibrations:
a global tangent-bundle obstruction on the total space, and zero
mixed curvature for the natural metric on the base.

The paper is organized as follows.
Section~\ref{ce335:sec:counterexample} proves
Theorem~\ref{ce335:thm:main}, including the metric lemma for cyclic
covers, the construction of the relative K\"ahler form, and the
verification of the tangent-bundle obstruction.
Section~\ref{sec:product-MA} establishes the product formula for
generalized Weil--Petersson metrics and its curvature consequences.
Section~\ref{sec:compact-Shimura} constructs the compact Shimura
families and proves
Theorems~\ref{thm:compact-Shimura-MA-example}
and~\ref{thm:compact-product-counterexample}.

\medskip
\noindent\textbf{Acknowledgments.}
Xueyuan Wan was supported by the National Key R\&D Program of China (Grant No.~2024YFA1013200) and the National Natural Science Foundation of China (Grant No.~12671100). The author used ChatGPT 5.6 and 6 Pro as auxiliary tools in this work.
The author completed and verified all mathematical arguments and
takes full responsibility for the content of this paper.\vspace{3mm}

\section{A counterexample among compact relative K\"ahler fibrations}
\label{ce335:sec:counterexample}

In this section, we prove Theorem~\ref{ce335:thm:main}. All varieties and morphisms in
this section are over $\mathbb C$. Two single-branch Hurwitz families
and a cyclic cover of degree five produce the underlying projective
fibration. We construct its relative K\"ahler form by gluing fiber
potentials, and then verify effectivity and the obstruction to
strict bisectional negativity on the total space. We begin with the
curvature and deformation facts used in these arguments.

\Needspace{6\baselineskip}
\subsection{Curvature and deformation preliminaries}
\label{ce335:sec:preliminaries}

\subsubsection*{Curvature conventions and two elementary consequences}

For a K\"ahler metric $g=(g_{i\bar j})$, we use the convention
\begin{equation}\label{ce335:eq:curvature-convention}
R_{i\bar j k\bar\ell}
=-\partial_k\partial_{\bar\ell}g_{i\bar j}
+g^{p\bar q}
 (\partial_k g_{i\bar q})(\partial_{\bar\ell}g_{p\bar j}).
\end{equation}
Thus the holomorphic bisectional curvature of $g$ is negative, that is, $\operatorname{HBC}(g)<0$,  means
\[
R(u,\bar u,v,\bar v)<0
\]
for every pair of nonzero $(1,0)$-vectors $u,v$ at the same point.
Equivalently, the induced Hermitian metric on the holomorphic tangent
bundle is Griffiths negative.

We use the Gauss equation and the curvature formula for holomorphic
subbundles with their induced metrics; see
\cite[Chapter~V, \S14]{Dem12}. In particular, a complex submanifold of a
K\"ahler manifold with strictly negative holomorphic bisectional curvature
inherits a K\"ahler metric with the same strict negativity.

\begin{lemma}\label{ce335:lem:graph}
Let $F:M\to N$ be holomorphic. Suppose $\omega_N$ is a K\"ahler metric
with nonpositive holomorphic bisectional curvature and $\omega_M$ is a
K\"ahler metric with strictly negative holomorphic bisectional curvature.
For every $a,b>0$, the form
\[
aF^*\omega_N+b\omega_M
\]
is a K\"ahler metric with strictly negative holomorphic bisectional
curvature, even if $F$ has critical points.
\end{lemma}

\begin{proof}
Use the holomorphic graph immersion
\[
(F,\operatorname{id}):M\longrightarrow
(N,a\omega_N)\times(M,b\omega_M).
\]
Its induced metric is the asserted sum. For nonzero $u,v\in T_xM$,
the ambient curvature evaluated on the corresponding graph vectors is
\[
\begin{aligned}
&aR^{\omega_N}\bigl(dF(u),\overline{dF(u)},dF(v),\overline{dF(v)}\bigr)
+bR^{\omega_M}(u,\bar u,v,\bar v)<0.
\end{aligned}
\]
The first term is nonpositive, including when an image vector vanishes,
and the second term is strictly negative. The Gauss equation subtracts
the squared norm of the second fundamental form, so strict negativity
holds for the induced metric as well.
\end{proof}

\begin{lemma}\label{ce335:lem:negative-vanishing}
Let $M$ carry a K\"ahler metric with strictly negative holomorphic
bisectional curvature, and let $Z\subset M$ be a compact complex
submanifold of positive dimension. Then
\[
H^0(Z,T_M|_Z)=0.
\]
Consequently, if $\varpi:M\to B$ is a proper holomorphic submersion
with positive-dimensional compact connected fibers, its Kodaira--Spencer
map is injective at every point.
\end{lemma}

\begin{proof}
The bundle $T_M|_Z$ is Griffiths negative in every nonzero tangent
direction of $Z$. For a holomorphic section $s$, the usual pointwise
Bochner identity gives, in any nonzero direction $u\in T_zZ$,
\[
\partial_u\partial_{\bar u}|s|^2
=|\nabla'_u s|^2-
 \bigl\langle R^{T_M|_Z}(u,\bar u)s,s\bigr\rangle.
\]
If $s$ is nonzero, at a positive maximum of $|s|^2$ the right-hand side
is strictly positive and the left-hand side is nonpositive. This is
impossible.

For the final assertion, restrict the tangent sequence to a fiber:
\[
0\longrightarrow T_{M_b}\longrightarrow T_M|_{M_b}
\longrightarrow\mathcal{O}_{M_b}\otimes T_bB\longrightarrow0.
\]
Since $M_b$ is compact and connected,
$H^0(M_b,\mathcal{O}_{M_b})=\mathbb{C}$. The connecting homomorphism is
$\rho_{\varpi,b}$, and the preceding vanishing makes it injective.
\end{proof}

\begin{lemma}\label{ce335:lem:tangent-obstruction}
Suppose a complex manifold $M$ contains a smooth compact curve $G$ and
a holomorphic line subbundle
\[
\mathcal{O}_G\hookrightarrow T_M|_G.
\]
Then $M$ admits no K\"ahler metric of strictly negative holomorphic
bisectional curvature.
\end{lemma}

\begin{proof}
Such an inclusion supplies a nonzero holomorphic section of $T_M|_G$,
contradicting Lemma~\ref{ce335:lem:negative-vanishing} if the asserted metric exists.
\end{proof}

\subsubsection*{Finite morphisms and Kodaira--Spencer classes}

\begin{lemma}\label{ce335:lem:trace-KS}
Let $r:Z\to H$ be a finite flat morphism of degree $m$ between smooth
projective varieties. For every vector bundle $E$ on $H$ and every $j$,
\[
r^*:H^j(H,E)\longrightarrow H^j(Z,r^*E)
\]
is injective.
Suppose, moreover, that $r$ is the fiber at $b\in B$ of a finite flat
morphism between smooth projective families over $B$,
\[
\begin{tikzcd}[column sep=large,row sep=large]
\mathscr Z\arrow[r,"\widetilde r"]\arrow[dr,"\varpi_Z"']&
\mathscr H\arrow[d,"\varpi_H"]\\
&B.
\end{tikzcd}
\]
Then, for $v\in T_bB$,
\begin{equation}\label{ce335:eq:KS-naturality}
H^1(dr)\bigl(\rho_{\varpi_Z,b}(v)\bigr)
=r^*\bigl(\rho_{\varpi_H,b}(v)\bigr)
\quad\text{in }H^1(Z,r^*T_H).
\end{equation}
In particular, injectivity of $\rho_{\varpi_H,b}$ implies injectivity
of $\rho_{\varpi_Z,b}$. Here $H^1(dr):H^1(Z,T_Z)\to H^1(Z,r^*T_H)$ denotes the induced map of $dr:T_Z\to r^*T_H$ on sheaf cohomology.
\end{lemma}
\begin{proof}
The cohomological assertion follows from \cite[Theorem 3.1(c)]{Wells}.
We now establish the compatibility of Kodaira--Spencer classes.
Identify the fibers over $b$ with
\(
Z=\varpi_Z^{-1}(b),
\) \(
H=\varpi_H^{-1}(b),
\)
so that $r=\widetilde r|_Z:Z\to H$. The commutativity of the
diagram means that
\(
\varpi_H\circ\widetilde r=\varpi_Z.
\)

Since $\varpi_Z$ and $\varpi_H$ are smooth, their tangent
sequences restricted to the fibers are
\[
0\longrightarrow T_Z
\longrightarrow T_{\mathscr Z}|_Z
\xrightarrow{\,d\varpi_Z\,}
\mathcal O_Z\otimes_{\mathbb C}T_bB
\longrightarrow0
\]
and
\[
0\longrightarrow T_H
\longrightarrow T_{\mathscr H}|_H
\xrightarrow{\,d\varpi_H\,}
\mathcal O_H\otimes_{\mathbb C}T_bB
\longrightarrow0.
\]
Let $\delta_Z$ and $\delta_H$ denote their connecting
homomorphisms. The Kodaira--Spencer maps are defined by
\[
\rho_{\varpi_Z,b}(v)=\delta_Z(1\otimes v),
\qquad
\rho_{\varpi_H,b}(v)=\delta_H(1\otimes v),
\qquad v\in T_bB.
\]

The second tangent sequence is locally split, so pulling it
back by $r$ gives an exact sequence on $Z$. Differentiating
$\varpi_H\circ\widetilde r=\varpi_Z$ then yields the following
commutative diagram of exact sequences on $Z$:
\[
\begin{tikzcd}[column sep=small,row sep=large]
0 \arrow[r] &
T_Z
  \arrow[r]
  \arrow[d,"dr"'] &
T_{\mathscr Z}|_Z
  \arrow[r,"d\varpi_Z"]
  \arrow[d,"d\widetilde r|_Z"] &
\mathcal O_Z\otimes_{\mathbb C}T_bB
  \arrow[r]
  \arrow[d,equal] &
0
\\
0 \arrow[r] &
r^*T_H
  \arrow[r] &
r^*(T_{\mathscr H}|_H)
  \arrow[r,"r^*(d\varpi_H)"'] &
\mathcal O_Z\otimes_{\mathbb C}T_bB
  \arrow[r] &
0.
\end{tikzcd}
\]
The middle vertical arrow takes values in
$r^*(T_{\mathscr H}|_H)$ because $\widetilde r|_Z=r$.
The right-hand vertical arrow is the identity because
$\widetilde r$ preserves the parameter in $B$.

Write
\[
\delta_{r^*H}:
H^0\bigl(Z,\mathcal O_Z\otimes_{\mathbb C}T_bB\bigr)
\longrightarrow H^1(Z,r^*T_H)
\]
for the connecting homomorphism of the bottom row.
Naturality of connecting homomorphisms gives
\[
H^1(dr)\bigl(\delta_Z(1\otimes v)\bigr)
=
\delta_{r^*H}(1\otimes v).
\]

We also have
\[
\delta_{r^*H}(1\otimes v)
=
r^*\bigl(\delta_H(1\otimes v)\bigr).
\]
To see this explicitly, choose an affine open cover
$\{U_i\}$ of $H$ on which the target tangent sequence splits,
and choose local holomorphic lifts
\(
W_i\in H^0(U_i,T_{\mathscr H}|_H)
\)
of the constant section $1\otimes v$. The differences
$W_j-W_i$ take values in $T_H$ and represent
$\delta_H(1\otimes v)$. On the inverse-image cover
$\{r^{-1}(U_i)\}$ of $Z$, the sections $r^*W_i$ are lifts of
$1\otimes v$ in the pulled-back sequence. Their differences are
$r^*(W_j-W_i)$, which represent both
$\delta_{r^*H}(1\otimes v)$ and
$r^*(\delta_H(1\otimes v))$.

Combining these identities with the definitions of the
Kodaira--Spencer maps gives
\[
\begin{aligned}
H^1(dr)\bigl(\rho_{\varpi_Z,b}(v)\bigr)
&=H^1(dr)\bigl(\delta_Z(1\otimes v)\bigr)\\
&=\delta_{r^*H}(1\otimes v)\\
&=r^*\bigl(\delta_H(1\otimes v)\bigr)\\
&=r^*\bigl(\rho_{\varpi_H,b}(v)\bigr)
\end{aligned}
\]
in $H^1(Z,r^*T_H)$. This proves
\eqref{ce335:eq:KS-naturality}.

Finally, suppose that $\rho_{\varpi_H,b}$ is injective and that
$v\in T_bB$ satisfies $\rho_{\varpi_Z,b}(v)=0$. The identity just
proved implies
\(
r^*\bigl(\rho_{\varpi_H,b}(v)\bigr)=0.
\)
By the first part of the proof, applied to $E=T_H$ and $j=1$,
the map
\(
r^*:H^1(H,T_H)\to H^1(Z,r^*T_H)
\)
is injective. Thus $\rho_{\varpi_H,b}(v)=0$, and the assumed
injectivity of $\rho_{\varpi_H,b}$ gives $v=0$. Hence
$\rho_{\varpi_Z,b}$ is injective.
\end{proof}

We will also use the following first-order rigidity observation. If
$r:Z\to H$ is a finite morphism of smooth projective curves and
$g(H)\ge2$, then
\begin{equation}\label{ce335:eq:map-rigidity}
H^0(Z,r^*T_H)=0,
\quad
\deg(r^*T_H)=\deg(r)(2-2g(H))<0.
\end{equation}
%Suppose that the Kodaira--Spencer classes of both curve families
%vanish in the parameter direction under consideration. Then
%suitable first-order changes of local coordinates remove the
%first-order variation of their transition maps. Thus both curves
%can be regarded as fixed to first order. The remaining variation
%of $r$ describes how the image $r(z)$ moves in $H$ for each fixed
%$z\in Z$, and therefore defines a holomorphic section of $r^*T_H$.
%Since $H^0(Z,r^*T_H)=0$, this variation must vanish. This argument
%only concerns first-order changes at the chosen parameter; it
%does not imply that nearby fibers are all biholomorphic.

\subsubsection*{The curve-fibration input}

The existence result below was first established by To and Yeung
\cite[Remark~1]{TY11}. Wan \cite[Corollary~1.4]{Wan26} subsequently gave
another proof by a direct metric construction. We state the
result in the form needed here.

\begin{theorem}\label{ce335:thm:curve-input}
Let $\varpi:M\to B$ be a proper holomorphic submersion between
compact complex manifolds, with connected fibers that are curves
of genus at least two. Suppose that its Kodaira--Spencer map is
injective everywhere and that $B$ admits a K\"ahler metric of
strictly negative holomorphic bisectional curvature. Then $M$
admits a K\"ahler metric of strictly negative holomorphic
bisectional curvature.
\end{theorem}

In Wan's construction, the fiberwise hyperbolic metrics determine
a relative K\"ahler form $\omega_{\mathrm{rel}}$ on $M$. If
$\omega_B$ is a K\"ahler form on $B$ with strictly negative
holomorphic bisectional curvature, then
\(
\omega_k=\omega_{\mathrm{rel}}+k\,\varpi^*\omega_B
\)
is a K\"ahler form with strictly negative holomorphic bisectional
curvature on $M$ for all sufficiently large $k$.

\subsection{A metric lemma for cyclic covers of surfaces}
\label{ce335:sec:metric-lemma}

\begin{proposition}\label{ce335:prop:cyclic-metric}
Let $(V,\omega_V)$ be a compact K\"ahler surface with
$\operatorname{HBC}(\omega_V)<0$. Let
\(
D=\bigsqcup_{i=1}^s D_i\subset V
\)
be a nonempty disjoint union of smooth connected curves satisfying
$D_i^2<0$. Suppose
\(
r:Z\to V
\)
is a smooth cyclic cover of degree five, totally ramified along $D$ and
unramified elsewhere. Then there is a function
$\varphi\in C^\infty(Z,\mathbb R)$ such that
\[
\omega_Z=r^*\omega_V+\sqrt{-1}\partial\bar\partial\varphi
\]
is a K\"ahler form with strictly negative holomorphic bisectional
curvature. In particular, $[\omega_Z]=r^*[\omega_V]$.
\end{proposition}

\begin{proof}
Let
\(
\mathcal R_i:=\bigl(r^{-1}(D_i)\bigr)_{\mathrm{red}},
\,
\mathcal R:=\bigsqcup_i\mathcal R_i.
\)
Thus $\mathcal R$ is the reduced ramification divisor: each
ramification curve is included with multiplicity one. Since $r$
has degree five and is totally ramified along each $D_i$, the
curve $\mathcal R_i$ is smooth and the restriction
\(
r|_{\mathcal R_i}:\mathcal R_i\to D_i
\)
is an isomorphism. Locally, one can choose coordinates such that
\(
r(z,w)=(z,w^5),
\,
\mathcal R_i=\{w=0\},
\,
D_i=\{y=0\},
\)
where $y$ is the second coordinate on $V$. The defining function
$y$ of $D_i$ therefore pulls back to $w^5$, which vanishes to
order five along $\mathcal R_i$. Hence, as divisors on $Z$,
\(
r^*D_i=5\mathcal R_i.
\)

Denote the holomorphic normal line bundles of these curves by
\(
N_{\mathcal R_i/Z}
:=T_Z|_{\mathcal R_i}/T_{\mathcal R_i}
\) and \(
N_{D_i/V}
:=T_V|_{D_i}/T_{D_i}.
\)
For a smooth curve in a smooth surface, its normal line bundle
is the restriction of the associated divisor line bundle, see \cite[Page 147]{MR1288523}. Thus
\(
N_{\mathcal R_i/Z}
\simeq\mathcal O_Z(\mathcal R_i)|_{\mathcal R_i},
\) \(
N_{D_i/V}
\simeq\mathcal O_V(D_i)|_{D_i}.
\)
Using these identifications and $r^*D_i=5\mathcal R_i$, we obtain
\begin{equation}\label{ce335:eq:normal-fifth}
\begin{aligned}
N_{\mathcal R_i/Z}^{\otimes5}
&\simeq\mathcal O_Z(5\mathcal R_i)|_{\mathcal R_i}\\
&\simeq\bigl(r^*\mathcal O_V(D_i)\bigr)|_{\mathcal R_i}\\
&\simeq(r|_{\mathcal R_i})^*
       \bigl(\mathcal O_V(D_i)|_{D_i}\bigr)\\
&\simeq(r|_{\mathcal R_i})^*N_{D_i/V}.
\end{aligned}
\end{equation}
Since $r|_{\mathcal R_i}$ is an isomorphism, taking degrees gives
\[
5\deg N_{\mathcal R_i/Z}
=\deg N_{D_i/V}
=D_i^2,
\]
where $D_i^2$ denotes the self-intersection number of $D_i$ in
$V$. Consequently, the assumption $D_i^2<0$ implies
\begin{equation}\label{ce335:eq:ram-normal-degree}
\deg N_{\mathcal R_i/Z}=\frac{D_i^2}{5}<0.
\end{equation}

\medskip
\noindent\emph{Choice of a normal metric.}
A line bundle of negative degree on a compact curve has a smooth
Hermitian metric with strictly negative curvature at every point.
 Choose such a metric on
each $N_{\mathcal{R}_i/Z}$.

Using
\(
\mathcal{O}_Z(\mathcal{R})|_{\mathcal{R}_i}\simeq N_{\mathcal{R}_i/Z},
\)
extend these metrics on $N_{\mathcal{R}_i/Z}$ to a smooth Hermitian metric $h_{\mathrm{ext}}$ on
$\mathcal{O}_Z(\mathcal{R})$. The covering group $\mu_5$ acts
naturally on this line bundle. Replace $h_{\mathrm{ext}}$ by its average
\(
h=\frac{1}{5}\sum_{\gamma\in\mu_5}\gamma^*h_{\mathrm{ext}}.
\)
This is a smooth $\mu_5$-invariant Hermitian metric. Each covering
transformation fixes $\mathcal R$ pointwise and acts on its
normal lines by multiplication by a fifth root of unity.
Such multiplication preserves every Hermitian norm. Hence the
averaged metric agrees with the prescribed normal-bundle metric
along $\mathcal R$.

Let $s_{\mathcal R}$ be the canonical holomorphic section of
$\mathcal O_Z(\mathcal R)$, whose zero divisor is $\mathcal R$,
and define
\(
\rho:=|s_{\mathcal R}|_h^2.
\)
The canonical section is preserved by the natural action on
$\mathcal O_Z(\mathcal R)$, and $h$ is invariant. Consequently,
$\rho$ is a smooth, nonnegative, $\mu_5$-invariant function on
$Z$, vanishing precisely along $\mathcal R$.

For $K>0$, set
\begin{equation}\label{ce335:eq:local-metric}
\omega_K=r^*\omega_V+\sqrt{-1}\partial\bar\partial(\rho+K\rho^2).
\end{equation}
We first prove that, for a suitable $K$, this form is a negatively
curved K\"ahler metric on a neighborhood of $\mathcal{R}$.

\medskip
\noindent\emph{The metric and mixed curvature on the ramification curve.}
Near a point of $\mathcal{R}$, choose coordinates in which
\[
r(z,w)=(z,w^5),\qquad \mathcal{R}=(w=0),
\qquad \rho=h(z,w)|w|^2.
\]
Here $h$ also denotes the positive local coefficient of the line-bundle
metric. Write
\[
g_0(z)=(\omega_V)_{z\bar z}(z,0),\qquad h_0(z)=h(z,0).
\]
Along $w=0$, the Hermitian matrix of \eqref{ce335:eq:local-metric} is
\begin{equation}\label{ce335:eq:metric-diagonal}
\begin{pmatrix}g_0(z)&0\\0&h_0(z)\end{pmatrix}.
\end{equation}
Thus it is positive definite there, independently of $K$.

We compute the mixed curvature along the ramification curve
$\mathcal R=\{w=0\}$. Since
$h(z,\zeta_5w)=h(z,w)$, the Taylor expansion of $h$ contains
no terms linear in $w$ or $\bar w$. Thus
\[
h(z,w)=h_0(z)+O(|w|^2),
\quad
\rho+K\rho^2=h_0(z)|w|^2+O(|w|^4).
\]
The pullback term is also easy to control. Since
$r(z,w)=(z,w^5)$, its relevant coefficients satisfy
\(
(r^*\omega_V)_{z\bar z}=g_0+O(|w|^5),
\) \(
(r^*\omega_V)_{z\bar w}=O(|w|^4).
\)
Consequently, the coefficients of $\omega_K$ have the expansions
\[
\begin{aligned}
g_{z\bar z}
&=g_0+(\partial_z\partial_{\bar z}h_0)|w|^2
  +O(|w|^4),\\
g_{z\bar w}
&=(\partial_z h_0)w+O(|w|^3),\\
g_{w\bar z}
&=(\partial_{\bar z}h_0)\bar w+O(|w|^3).
\end{aligned}
\]
Here the remainder terms are understood in the smooth Taylor
sense.

Evaluating the derivatives needed for the curvature formula
at $w=0$, we obtain
\[
\begin{gathered}
\partial_w\partial_{\bar w}g_{z\bar z}
=\partial_z\partial_{\bar z}h_0,
\qquad
\partial_w g_{z\bar z}
=\partial_{\bar w}g_{z\bar z}=0,\\
\partial_w g_{z\bar w}=\partial_z h_0,
\qquad
\partial_{\bar w}g_{w\bar z}=\partial_{\bar z}h_0.
\end{gathered}\]
By \eqref{ce335:eq:metric-diagonal}, the inverse metric at
$w=0$ is $\operatorname{diag}(g_0^{-1},h_0^{-1})$.
Thus only the term with $p=q=w$ contributes to the sum in
\eqref{ce335:eq:curvature-convention}, and we find
\begin{equation}\label{ce335:eq:mixed-curvature}
\begin{aligned}
R_{z\bar z w\bar w}
&=-\partial_w\partial_{\bar w}g_{z\bar z}
  +\frac{1}{h_0}
   (\partial_w g_{z\bar w})
   (\partial_{\bar w}g_{w\bar z})\\
&=-\partial_z\partial_{\bar z}h_0
  +\frac{(\partial_z h_0)(\partial_{\bar z}h_0)}{h_0}\\
&=-h_0\partial_z\partial_{\bar z}\log h_0.
\end{aligned}
\end{equation}

The function $h_0$ is the local coefficient of the prescribed
Hermitian metric on the normal line bundle. Its curvature
coefficient is
$-\partial_z\partial_{\bar z}\log h_0$, which is strictly
negative by our choice of that metric. To express the mixed
curvature in unit directions, take
\(
e_1=g_0^{-1/2}\partial_z,
\,
e_2=h_0^{-1/2}\partial_w.
\)
Then
\begin{equation}\label{ce335:eq:c-negative}
c:=R(e_1,\bar e_1,e_2,\bar e_2)
=\frac{R_{z\bar z w\bar w}}{g_0h_0}
=-\frac{\partial_z\partial_{\bar z}\log h_0}{g_0}<0.
\end{equation}
Finally, the term $K\rho^2$ begins at order $|w|^4$ in the
potential. It therefore changes neither the metric at $w=0$
nor the derivatives used in the mixed-curvature calculation
above. Hence $c$ is independent of $K$ along $\mathcal R$.

\medskip
\noindent\emph{Tangential and normal curvature.}
The form $\omega_K$ is invariant under the deck transformation
$(z,w)\mapsto(z,\zeta_5w)$. Its fixed curve $\mathcal{R}$ is totally geodesic:
an isometry fixes the geodesic with prescribed initial data tangent to
its fixed locus, by uniqueness of geodesics. The restriction of
$\omega_K$ to $\mathcal{R}_i$ is precisely the pullback of
$\omega_V|_{D_i}$. Hence the unit tangential curvature coefficient
\(
a=R_{1\bar1 1\bar1}
\)
is strictly negative, by the Gauss equation for $D_i\subset V$.

Let $b_K=R_{2\bar2 2\bar2}$ be the unit normal coefficient, and let
$b_0$ denote the coefficient for $K=0$. The function $\rho^2$ vanishes
to order four in the normal variable. Consequently, adding
$\sqrt{-1}\partial\bar\partial(K\rho^2)$ changes neither the metric nor its first derivatives
on $\mathcal{R}$. Since
\[
\partial_w^2\partial_{\bar w}^2(K\rho^2)\big|_{w=0}
=4Kh_0^2,
\]
\eqref{ce335:eq:curvature-convention} gives the exact identity
\begin{equation}\label{ce335:eq:normal-curvature}
b_K=b_0-4K.
\end{equation}
The function $b_0$ is bounded on the compact curve $\mathcal{R}$. Fix $K$ so
large that $b_K<0$ everywhere on $\mathcal{R}$.

\medskip
\noindent\emph{Control of every bisectional curvature component.}
At a point of $\mathcal{R}$, take a unit tangential vector $e_1$ and a unit
normal vector $e_2$. The differential of a generator of $\mu_5$ acts as
$\operatorname{diag}(1,\zeta_5)$. In a component
$R_{i\bar j k\bar\ell}$, the difference between the number of
holomorphic normal indices and antiholomorphic normal indices is an
integer between $-2$ and $2$. Invariance forces the component to vanish
unless this difference is zero. K\"ahler symmetries then leave precisely
\[
\begin{gathered}
R_{1\bar1 1\bar1}=a,\qquad R_{2\bar2 2\bar2}=b_K,\\
R_{1\bar1 2\bar2}=R_{1\bar2 2\bar1}
=R_{2\bar1 1\bar2}=R_{2\bar2 1\bar1}=c.
\end{gathered}
\]
Thus, for $u=u_1e_1+u_2e_2$ and $v=v_1e_1+v_2e_2$,
\begin{equation}\label{ce335:eq:all-bisectional}
R(u,\bar u,v,\bar v)
=a|u_1v_1|^2+b_K|u_2v_2|^2
 +c|u_1v_2+u_2v_1|^2.
\end{equation}
All three coefficients are strictly negative. The elementary inequality
\begin{equation}\label{ce335:eq:squares}
|u_1v_1|^2+|u_2v_2|^2+|u_1v_2+u_2v_1|^2
\ge\frac12\|u\|^2\|v\|^2
\end{equation}
follows by expanding the last square and using
$2|u_1u_2|\le\|u\|^2$ and $2|v_1v_2|\le\|v\|^2$.
Compactness gives a constant $\lambda>0$ with
$a,b_K,c\le-\lambda$ on $\mathcal{R}$. Therefore
\[
R(u,\bar u,v,\bar v)
\le-\frac{\lambda}{2}\|u\|^2\|v\|^2
\quad\text{on }\mathcal{R}.
\]
By continuity and \eqref{ce335:eq:metric-diagonal}, there is a fixed open
neighborhood $U$ of $\mathcal{R}$ on which $\omega_K$ is K\"ahler and has
strictly negative holomorphic bisectional curvature.

\medskip
\noindent\emph{Global patching.}
Choose an open neighborhood $U_0$ with
$\mathcal{R}\subset U_0\Subset U$, and a smooth function
$0\le\chi\le1$ that equals $1$ on $U_0$ and has compact support in $U$.
For $0<\varepsilon<1$, define the global closed real $(1,1)$-form
\begin{equation}\label{ce335:eq:global-metric}
\omega_\varepsilon
=r^*\omega_V+
 \varepsilon\sqrt{-1}\partial\bar\partial\bigl(\chi(\rho+K\rho^2)\bigr).
\end{equation}
On $U_0$ this is exactly
\[
\omega_\varepsilon
=(1-\varepsilon)r^*\omega_V+\varepsilon\omega_K.
\]
It is K\"ahler with strictly negative holomorphic bisectional curvature
by Lemma~\ref{ce335:lem:graph}, applied to $r|_{U_0}:U_0\to V$. This argument
remains valid at the ramification curve, where $r^*\omega_V$ is
degenerate.

The compact set $\operatorname{Supp}\chi\setminus U_0$ is disjoint from $\mathcal{R}$.
There $r^*\omega_V$ is already a nondegenerate negatively curved
K\"ahler metric. The perturbation in \eqref{ce335:eq:global-metric} tends to
zero in the $C^2$ norm of the metric as $\varepsilon\to0$.
Positivity and strict negativity of bisectional curvature consequently
persist there for all sufficiently small $\varepsilon>0$. Outside
$\operatorname{Supp}\chi$ the metric is unchanged.

The order of choices is $h$, then $K$, then $U,U_0,\chi$, and finally
$\varepsilon$. Thus the compact transition region is fixed before the
small-perturbation argument is used. Formula
\eqref{ce335:eq:global-metric} also proves
$[\omega_\varepsilon]=r^*[\omega_V]$.
\end{proof}

%\begin{remark}\label{ce335:rem:degree-versus-metric}
%The proof needs a normal metric with strictly negative curvature at
%\emph{every} point. A pullback of a negatively curved line-bundle
%metric under a ramified map of curves need not have this property.
%The negative degree in \eqref{ce335:eq:ram-normal-degree}, followed by the
%choice of a new metric, is what supplies the required strictness.
%\end{remark}

\subsection{From fiber metrics to a relative K\"ahler form}
\label{ce335:sec:relative-negative-form}

Formula~\eqref{ce335:eq:global-metric} provides more than the existence
of a negatively curved metric on a single covering surface: the metric
is obtained from a prescribed pullback form by adding a global
$\partial\bar\partial$-potential. This additional feature allows the
fiber metrics to be assembled into a closed form on the total space.
We record the relevant gluing statement.

\begin{proposition}\label{ce335:prop:relative-negative-form}
Let $p:M\to B$ be a compact holomorphic fibration with
positive-dimensional fibers, and let $\eta$ be a
smooth real $d$-closed $(1,1)$-form on $M$. Suppose that, for every
$b\in B$, there is a function $\varphi_b\in C^\infty(M_b,\mathbb R)$
such that
\(
\eta|_{M_b}+\sqrt{-1}\partial_b\bar\partial_b\varphi_b
\)
is a K\"ahler form with strictly negative holomorphic bisectional
curvature. Then there exists $\Phi\in C^\infty(M,\mathbb R)$ such that
\begin{equation}\label{ce335:eq:relative-negative-form}
\omega=\eta+\sqrt{-1}\partial\bar\partial\Phi
\end{equation}
is a relative K\"ahler form and $\omega|_{M_b}$ has strictly negative
holomorphic bisectional curvature for every $b\in B$. In particular,
$[\omega]=[\eta]$ in $H^2(M,\mathbb R)$.

If $B$ admits a K\"ahler form $\omega_B$, then
$\omega+k p^*\omega_B$ is a K\"ahler form on $M$ for all sufficiently
large $k$. Its restriction to every fiber is still $\omega|_{M_b}$.
\end{proposition}

\begin{proof}
We first explain why constant positive linear combinations preserve
strictly negative holomorphic bisectional curvature. Let
$g_1,\ldots,g_r$ be K\"ahler metrics with this property on a complex
manifold $F$, and let $\lambda_1,\ldots,\lambda_r$ be positive real
constants. The diagonal map
\(
\delta:(F,\sum_{j=1}^r\lambda_jg_j)
\to\prod_{j=1}^r(F,\lambda_jg_j)
\)
is a holomorphic isometric immersion. If $g=\sum_j\lambda_jg_j$ and
$u,v\in T_x^{1,0}F$ are nonzero, the Gauss equation gives
\begin{equation}\label{ce335:eq:convex-negative-curvature}
R^g(u,\bar u,v,\bar v)
=\sum_{j=1}^r\lambda_jR^{g_j}(u,\bar u,v,\bar v)
-\|\mathrm{II}(u,v)\|^2<0.
\end{equation}
Here $\mathrm{II}$ is the second fundamental form of $\delta$, with
the norm induced by the product metric. This is the diagonal case of
the argument in Lemma~\ref{ce335:lem:graph}. Nonnegative coefficients
are also allowed, provided at least one is positive, by omitting the
zero terms.

Fix $b_0\in B$. The compact fiber $M_{b_0}$ is a closed smooth
submanifold of $M$, so $\varphi_{b_0}$ extends to a real smooth function
$\Phi_{b_0}$ on $M$. For example, one may extend it using a tubular
neighborhood and multiply by a cutoff function that is identically
one near the fiber. Set
\[
\eta_{b_0}=\eta+\sqrt{-1}\partial\bar\partial\Phi_{b_0}.
\]
Its restriction to $M_{b_0}$ is the prescribed negatively curved
K\"ahler form. In local submersion coordinates, the coefficients of
the metrics induced by $\eta_{b_0}$ on the fibers, together with their
fiber derivatives, depend smoothly on the parameter. Positivity of the K\"ahler form and
strict negativity of holomorphic bisectional curvature are open
conditions. Compactness of $M_{b_0}$ therefore gives an open
neighborhood $W$ of the entire fiber on which the induced fiber
metrics satisfy both conditions. 
%More explicitly, one first retains
%positivity of the vertical Hermitian form and then uses the curvature
%formula on the compact bundles of pairs of unit vertical tangent
%vectors. The strict upper bound on the curvature along $M_{b_0}$
%persists after shrinking $W$.

Since $p$ is proper, $p(M\setminus W)$ is closed and does not contain
$b_0$. Thus
\(
U_{b_0}=B\setminus p(M\setminus W)
\)
is an open neighborhood of $b_0$ with $p^{-1}(U_{b_0})\subset W$.
It follows that $\eta_{b_0}|_{M_b}$ is K\"ahler with strictly negative
holomorphic bisectional curvature for every $b\in U_{b_0}$.

Choose finitely many such neighborhoods $U_1,\ldots,U_N$ covering
$B$, with corresponding global potentials $\Phi_1,\ldots,\Phi_N$.
Write
\[
\eta_a=\eta+\sqrt{-1}\partial\bar\partial\Phi_a.
\]
Let $\{\chi_a\}_{a=1}^N$ be a smooth nonnegative partition of unity
on $B$ subordinate to this cover, with
$\operatorname{supp}\chi_a\subset U_a$. Define
\[
\Phi=\sum_{a=1}^N(\chi_a\circ p)\Phi_a,
\qquad
\omega=\eta+\sqrt{-1}\partial\bar\partial\Phi.
\]
The form $\omega$ is real and $d$-closed. For fixed $b\in B$, the
functions $\chi_a\circ p$ are constant on $M_b$, so restriction to
the fiber gives
\begin{equation}\label{ce335:eq:relative-convex-restriction}
\begin{aligned}
\omega|_{M_b}
&=\eta|_{M_b}
 +\sqrt{-1}\partial_b\bar\partial_b
  (\sum_a\chi_a(b)\Phi_a|_{M_b})=\sum_a\chi_a(b)\eta_a|_{M_b}.
\end{aligned}
\end{equation}
Whenever $\chi_a(b)>0$, one has $b\in U_a$, and hence
$\eta_a|_{M_b}$ is a K\"ahler form with strictly negative holomorphic
bisectional curvature. The coefficients in
\eqref{ce335:eq:relative-convex-restriction} are constants on $M_b$,
are nonnegative, and sum to one. The diagonal argument above proves
the required strict negativity. 
%Notice that we have glued the
%potentials rather than the forms themselves; this preserves closedness
%on $M$ while giving a constant convex combination on each fiber.
The cohomology assertion follows directly from
\eqref{ce335:eq:relative-negative-form}.

For the last assertion, choose a smooth complex splitting
$T_M=T_{M/B}\oplus H$. In this splitting the Hermitian forms associated
with $\omega$ and $p^*\omega_B$ have block matrices
\[
\omega=\begin{pmatrix}A&D\\D^*&E\end{pmatrix},
\qquad
p^*\omega_B=\begin{pmatrix}0&0\\0&G\end{pmatrix},
\]
where $A$ and $G$ are positive definite. By compactness, their
smallest eigenvalues, measured using fixed background Hermitian
metrics, have positive lower bounds. The Schur complement of $A$ in
$\omega+k p^*\omega_B$ is
\(
E+kG-D^*A^{-1}D.
\)
All terms other than $kG$ are uniformly bounded on $M$, so this
complement is positive definite for all sufficiently large $k$.
Thus $\omega+k p^*\omega_B$ is positive definite and closed. Finally,
$p^*\omega_B$ restricts to zero on every fiber.
\end{proof}

\subsection{Single-point triple covers and their families}
\label{ce335:sec:Hurwitz}

We use the iterated Kodaira--Oort construction in the form recalled in
\cite[Lemma~3.1 and \S3]{Jab09}. We include the family version because the argument requires
an actual family of covering maps after a finite
\'etale change of parameter, together with pointwise
Kodaira--Spencer injectivity.

We record the degree-three case of the construction in
\cite[Section~5.2]{DonagiWitten2015}, together with its
immediate consequences.

\begin{lemma}\label{ce335:lem:triple-point}
Let $H$ be a smooth projective complex curve of genus $g\ge2$,
and let $x\in H$. There exists a smooth connected projective
curve $Z$ and a finite morphism
\[
r:Z\longrightarrow H
\]
of degree three, totally ramified over $x$ and unramified
over $H\setminus\{x\}$. Its monodromy group is $S_3$,
its automorphism group over $H$ is trivial, and
\[
g(Z)=3g-1.
\]
\end{lemma}
\begin{proof}
The existence, monodromy, and genus assertions are given by
\cite[Section~5.2]{DonagiWitten2015} with $d=3$.
The automorphism assertion follows because an automorphism
over $H$ acts on a general fiber by a permutation commuting
with the monodromy group, and the centralizer of $S_3$
in $S_3$ is trivial.
\end{proof}
The following lemma is an application of the theory of Hurwitz
spaces with prescribed monodromy; see
\cite[Theorem~5.1.5, Proposition~6.5.2, and Section~6.6]
{BertinRomagny2011}.
For a convenient formulation in the Galois case, see
\cite[Theorem~4.1]{LemosTorzewski2023}.

\begin{lemma}\label{ce335:lem:Hurwitz-family}
Let $\pi:\mathscr H\to U$ be a smooth projective family of
connected curves of genus $g\ge2$ over a smooth connected
projective complex variety $U$, and let $s:U\to\mathscr H$
be a section.
There exist a smooth connected projective variety $U'$,
a finite \'etale surjection $\tau:U'\to U$, and a finite
flat morphism of degree three
\[
\widetilde r:\mathscr Z\longrightarrow\mathscr H',
\quad
\mathscr H':=\mathscr H\times_U U'=\left\{\left(h, u^{\prime}\right) \in \mathscr{H} \times U^{\prime} \mid \pi(h)=\tau\left(u^{\prime}\right)\right\},
\]
such that $\mathscr Z\to U'$ is a smooth projective family
of connected curves.

Let $s':U'\to\mathscr H'$ be the pullback of $s$. For every
$u'\in U'$, the induced cover
\(
r_{u'}: Z_{u'} \to H_{\tau(u')}
\)
has monodromy group $S_3$, is totally ramified over
$s(\tau(u'))$, and is unramified elsewhere.
Moreover, the reduced ramification divisor
\(
\mathcal R
:=
\bigl(\widetilde r^{-1}(s'(U'))\bigr)_{\mathrm{red}}
\)
maps isomorphically to $s'(U')$. In particular,
$\mathcal R$ is the image of a section of
$\mathscr Z\to U'$.
\end{lemma}

\begin{proof}
For a smooth projective family of curves of genus at least two
over a complex variety $U$, with pairwise disjoint marked sections,
fix the degree and monodromy type of the covers, together with
their ramification at the marked sections, and require them to
be unramified elsewhere. By
\cite[Proposition~6.5.2(ii), Definition~6.6.5, and
Theorem~6.6.6]{BertinRomagny2011}, the corresponding relative Hurwitz stack
is proper, quasi-finite, and \'etale over $U$. If every cover has
trivial automorphism group over its target, this stack is
represented by a finite \'etale scheme $T \to U$ carrying
a universal family of covering maps. Isomorphisms here are
required to induce the identity on the target curve.

Apply this result to $(\mathscr H \to U,s)$ and to connected
degree-three covers with monodromy group $S_3$, ramification
type $(3)$ at the marked point, and no other branch points.
Thus each cover has exactly one point above the marked point,
where its local analytic equation is $z=w^3$.
Every such cover $r_u \colon Z_u \to H_u$ has trivial
automorphism group over $H_u$: an automorphism acts faithfully
on an unramified fiber and commutes with monodromy, so
\(
  \operatorname{Aut}_{H_u}(Z_u)
  \hookrightarrow C_{S_3}(S_3)=\{1\}.
\)
Consequently, there is a finite \'etale scheme $T \to U$
carrying a universal family of these covers. Its fiber over
$u$ parametrizes their isomorphism classes over $H_u$ and
is nonempty by Lemma~\ref{ce335:lem:triple-point}.
Choose a nonempty connected component $U'$ of $T$.
The restriction $\tau \colon U' \to U$ is finite \'etale,
so its image is both open and closed. Since $U$ is connected,
$\tau$ is surjective. Moreover, $U'$ is smooth because $\tau$
is \'etale, and projective because $\tau$ is finite and $U$
is projective.

Restricting the universal family to $U'$ gives
\(
\widetilde r:\mathscr Z\to\mathscr H'.
\)
By construction, its fibers are the prescribed connected
covers. We spell out the local behavior along the marked
section.
Write $n=\dim U'$. Near any point of $s'(U')$, choose
holomorphic coordinates $(z,t_1,\ldots,t_n)$ such that
the projection to $U'$ is given by $t=(t_1,\ldots,t_n)$
and the marked section is $z=0$. The prescribed
ramification type gives the local form
\(
(w,t_1,\ldots,t_n)
\mapsto
(w^3,t_1,\ldots,t_n).
\)
In these coordinates, the projection $\mathscr Z\to U'$ is
given by $(w,t)\mapsto t$, so it is smooth near the ramification
locus. Away from this locus, $\widetilde r$ is \'etale;
composing it with the smooth projection $\mathscr H'\to U'$
shows that $\mathscr Z\to U'$ is smooth there as well.

The same local model proves flatness. If
\(
A=\mathbb C\{z,t_1,\ldots,t_n\},
\)
then the local algebra of the cover is
\(
A[w]/(w^3-z),
\)
which is free over $A$ with basis $1,w,w^2$.
Away from the marked section, the map is finite \'etale
of degree three. Hence $\widetilde r$ is finite flat
of degree three everywhere.

Since $\mathscr H'$ is projective and $\widetilde r$
is finite, $\mathscr Z$ is projective. It follows that
$\mathscr Z\to U'$ is a smooth projective family with
connected fibers.

Finally, in the displayed local coordinates the reduced
ramification divisor is $w=0$, and its map to the marked
section is
\(
(0,t)\mapsto(0,t).
\)
These local isomorphisms show that
\(
\widetilde r|_{\mathcal R}:
\mathcal R\xrightarrow{\sim}s'(U').
\)
Composing with $s'(U')\xrightarrow{\sim}U'$ proves that
$\mathcal R$ is a section of $\mathscr Z\to U'$.
\end{proof}

\subsection{A negatively curved threefold with two fibrations}
\label{ce335:sec:tower}

\subsubsection*{The first family and its negative section}

Let $C$ be the smooth projective curve obtained by completing
the affine curve
\begin{equation}\label{ce335:eq:base-curve}
C:\quad y^2=x^6-1.
\end{equation}
The coordinate $x$ defines a holomorphic map
\(
p:C\to\mathbb P^1.
\)
For a general value of $x$, the equation has two distinct
solutions for $y$, so $p$ has degree two. The two solutions
coincide precisely when $x^6=1$. Since these six roots are simple, the double cover has
ramification index two over each of them. To examine the
points above infinity, set $u=1/x$ and $v=y/x^3$.
The equation becomes $v^2=1-u^6$, giving two points
$(u,v)=(0,\pm1)$ above infinity. At both points, $u$ is
a local coordinate, so the map to $\mathbb P^1$ is
unramified there.
Since $\mathbb P^1$ has genus zero, the Riemann--Hurwitz
formula \cite[Tag~0C1F]{Stacks} gives
\[
2g(C)-2
=2\bigl(2g(\mathbb P^1)-2\bigr)+6(2-1)
=-4+6=2.
\]
Hence $g(C)=2$.

Consider the second projection
\(
\operatorname{pr}_2:C\times C\to C.
\)
Its fiber over $c\in C$ is $C\times\{c\}$. The diagonal
section $c\mapsto(c,c)$ marks the point $c$ on this copy of $C$.
Applying Lemma~\ref{ce335:lem:Hurwitz-family} to this
family gives a finite \'etale cover
\(
\nu:B_0\to C
\)
and a finite flat morphism of degree three
\begin{equation}\label{ce335:eq:first-cover}
A=(a,q):S\longrightarrow C \times B_0\simeq (C \times C) \times_C B_0.
\end{equation}
Here $a$ and $q$ are the compositions of $A$ with the
projections to $C$ and $B_0$, respectively. The curve
$B_0$ and the surface $S$ are smooth, connected, and
projective.

The projection $q:S\to B_0$ is a smooth family of connected
curves. For each $t\in B_0$, the restriction
\(
a_t:S_t:=q^{-1}(t)\to C
\)
is a degree-three cover, totally ramified over $\nu(t)$
and unramified elsewhere. Since $g(C)=2$, the
Riemann--Hurwitz formula gives
$g(S_t)=3g(C)-1=5$.

As $t$ varies, the branch point traces out the graph
\[
\Gamma=\{(\nu(t),t):t\in B_0\}\subset C\times B_0,
\]
which is therefore the branch divisor of $A$.
By Lemma~\ref{ce335:lem:Hurwitz-family}, the unique ramification
point in each fiber varies holomorphically and defines a
section of $q:S\to B_0$. Denote its image by $R\subset S$; this is the
reduced ramification curve, and
$q|_R:R\to B_0$ is an isomorphism.
Since the ramification index along $R$ is three,
\begin{equation}\label{ce335:eq:triple-divisor}
A^*\Gamma=3R.
\end{equation}

Identify $\Gamma$ with $B_0$ via the graph embedding
$t\mapsto(\nu(t),t)$. The normal bundle of the graph is then
\(
N_{\Gamma/(C\times B_0)}\simeq \nu^*T_C.
\)
Moreover, $A$ maps the ramification curve $R$ isomorphically
onto $\Gamma$, and both curves are identified with $B_0$
by projection to the second factor.

To compute the normal bundle of $R$ in $S$, recall that
\(
N_{R/S}\simeq \mathcal O_S(R)|_R.
\)
\eqref{ce335:eq:triple-divisor} gives an isomorphism of line bundles
\(
\mathcal O_S(3R)
\simeq A^*\mathcal O_{C\times B_0}(\Gamma).
\)
Restricting this isomorphism to $R$, and using
$\mathcal O_{C\times B_0}(\Gamma)|_\Gamma
\simeq N_{\Gamma/(C\times B_0)}$, we obtain
\begin{equation}\label{ce335:eq:R-normal}
N_{R/S}^{\otimes3}
\simeq (A|_R)^*N_{\Gamma/(C\times B_0)}
\simeq \nu^*T_C,
\end{equation}
where the last isomorphism uses the identification $R\simeq B_0$.
The self-intersection number $R^2$ equals the degree of
$N_{R/S}$. Taking degrees in \eqref{ce335:eq:R-normal} therefore gives
\[
3R^2
=\deg(\nu^*T_C)
=(\deg\nu)\deg T_C
=(\deg\nu)(2-2g(C)).
\]
Since $g(C)=2$, it follows that
\begin{equation}\label{negative of R^2}
R^2=-\frac{2\deg\nu}{3}<0.
\end{equation}

We next show that the family $q:S\to B_0$ is effectively
parametrized, meaning that its Kodaira--Spencer map is
injective at every point. Fix $t\in B_0$ and suppose that
$v\in T_tB_0$ satisfies $\rho_{q,t}(v)=0$. Choose a local
parameter curve $t(s)$ in $B_0$ with $t(0)=t$ and $t'(0)=v$.

The vanishing of the Kodaira--Spencer class allows us to
identify the varying curves $S_{t(s)}$ with $S_t$ to first
order in $s$. Since the target curve $C$ is fixed, the
first-order variation of the covering maps
$a_{t(s)}:S_{t(s)}\to C$ then defines a holomorphic section
\(
\dot a\in H^0(S_t,a_t^*T_C).
\)
This section must vanish: indeed,
\(
\deg(a_t^*T_C)=3(2-2g(C))=-6<0,
\)
so $a_t^*T_C$ has no nonzero holomorphic sections.

To see what this implies for the branch value, let $p(s)$
be the unique ramification point of $a_{t(s)}$, and write
$p=p(0)$. By construction,
\(
a_{t(s)}(p(s))=\nu(t(s)).
\)
Differentiating this identity using the first-order
identification above gives
\[
d\nu_t(v)
=\dot a(p)+(da_t)_p(\dot p)=0.
\]
Here $\dot a=0$ by the preceding argument, and
$(da_t)_p=0$ because $p$ is a ramification point.
Since $\nu$ is \'etale, $d\nu_t$ is an isomorphism.
Consequently, $v=0$, proving that $\rho_{q,t}$ is injective.

Finally, the Riemann--Hurwitz formula for the unramified
cover $\nu:B_0\to C$ gives
\[
2g(B_0)-2
=(\deg\nu)(2g(C)-2)
=2\deg\nu.
\]
Thus $g(B_0)=1+\deg\nu\ge2$.
Consequently, Theorem~\ref{ce335:thm:curve-input} applied to $q$, gives a K\"ahler
metric of strictly negative holomorphic bisectional curvature on $S$.

\subsubsection*{The second family}

Consider the pointed family
\(
S\times_{B_0}S\to S,
\)
where the projection is to the second factor and the section is the
diagonal. Lemma~\ref{ce335:lem:Hurwitz-family} gives a
finite \'etale surjection
\(
\eta:T\to S
\)
and a degree-three morphism
\(
\mathcal B:Y\to S\times_{B_0}T.
\)
Write
$
p=\operatorname{pr}_T\circ\mathcal B:Y\to T,
$ $
b=\operatorname{pr}_S\circ\mathcal B:Y\to S.
$
Here $T$ is a smooth connected projective surface, $Y$ is a smooth
connected projective threefold, and $p$ is smooth with connected
fibers of genus $3\cdot5-1=14$ by Lemma \ref{ce335:lem:triple-point} and $g(S_t)=5$. For each $t\in T$,
\(
b_t:Y_t\to S_{q(\eta(t))}
\)
is totally ramified over $\eta(t)$ and unramified elsewhere.
For any $y\in Y$, $\mathcal{B}(y)=(b(y),p(y))\in S\times_{B_0}T$, by definition of $S\times_{B_0}T$, one has
$
q\circ b=(q\circ\eta)\circ p.
$

We now prove that $p:Y\to T$ is effectively parametrized.
Set $k:=q\circ\eta:T\to B_0$.
The maps satisfy $q\circ b=k\circ p$.

Fix $t\in T$, write $\beta=k(t)$, and suppose that
$\xi\in T_tT$ satisfies $\rho_{p,t}(\xi)=0$.
We must show that $\xi=0$.

First, we show that $dk_t(\xi)=0$. Applying
\eqref{ce335:eq:KS-naturality} to the degree-three covering
$b_t:Y_t\to S_\beta$ gives
\[
b_t^*\bigl(\rho_{q,\beta}(dk_t(\xi))\bigr)
=
H^1(db_t)\bigl(\rho_{p,t}(\xi)\bigr)
=0
\quad\text{in }H^1(Y_t,b_t^*T_{S_\beta}).
\]
Here $H^1(db_t)$ is induced by the differential
$T_{Y_t}\to b_t^*T_{S_\beta}$.
By Lemma~\ref{ce335:lem:trace-KS}, the pullback map $b_t^*$
is injective. Hence
\(
\rho_{q,\beta}(dk_t(\xi))=0.
\)
Since $q$ is effectively parametrized, its
Kodaira--Spencer map is injective, therefore
\(
dk_t(\xi)=0.
\)

It remains to examine the motion of the branch value.
Choose a local parameter curve $t(s)$ in $T$ with
$t(0)=t$ and $t'(0)=\xi$.
Since $dk_t(\xi)=0$, the parameter $k(t(s))$ is equal
to $\beta$ to first order. Thus the target curves
$S_{k(t(s))}$ are identified with the fixed curve
$S_\beta$ to first order. Moreover,
\[
dq_{\eta(t)}\bigl(d\eta_t(\xi)\bigr)=dk_t(\xi)=0,
\]
so the velocity $d\eta_t(\xi)$ of the branch value
lies in $T_{\eta(t)}S_\beta$.

The assumption $\rho_{p,t}(\xi)=0$ also allows us
to identify $Y_{t(s)}$ with $Y_t$ to first order.
Under these identifications, the first-order variation
of the covering maps defines a holomorphic section
\(
\dot b\in H^0(Y_t,b_t^*T_{S_\beta}).
\)
This section vanishes by \eqref{ce335:eq:map-rigidity},
since
\(
\deg(b_t^*T_{S_\beta})
=3(2-2g(S_\beta))=-24<0.
\)

Let $P(s)\in Y_{t(s)}$ be the unique ramification
point, and put $P=P(0)$. Its image is the marked
branch value:
\(
b_{t(s)}(P(s))=\eta(t(s)).
\)
Differentiating this identity using the first-order
identifications above gives
\[
d\eta_t(\xi)
=\dot b(P)+(db_t)_P(\dot P)=0.
\]
Indeed, $\dot b=0$, and $(db_t)_P=0$ because $P$
is a ramification point. Finally, $\eta$ is \'etale,
so $d\eta_t$ is an isomorphism. It follows that
$\xi=0$, proving that $\rho_{p,t}$ is injective.

The surface $T$ inherits a negatively curved K\"ahler metric from
$S$ by \'etale pullback. Applying Theorem~\ref{ce335:thm:curve-input} to $p$
therefore gives a K\"ahler metric of strictly negative holomorphic
bisectional curvature on $Y$.

% Required packages:
% \usepackage{amsmath,amssymb}
% \usepackage{tikz-cd}

\subsubsection*{The transverse surface fibration}

Define
\begin{equation}\label{ce335:eq:f-D-W}
f=a\circ b:Y\longrightarrow C,
\quad
D=\eta^{-1}(R)\subset T,
\quad
W=p^{-1}(D)\subset Y.
\end{equation}
Since $\eta$ is \'etale and $R$ is smooth, the divisor $D$ is a nonempty disjoint union of smooth curves.
For each connected component $D_i$, \'etale pullback and \eqref{negative of R^2} gives
\begin{equation}\label{ce335:eq:D-negative}
N_{D_i/T}\simeq(\eta|_{D_i})^*N_{R/S},
\quad
D_i^2=\deg(D_i\to R)\,R^2<0.
\end{equation}
Since $p$ is smooth, $W$ is a smooth reduced divisor, possibly
disconnected.

\begin{lemma}\label{ce335:lem:transversality}
Both $f:Y\to C$ and $f|_W:W\to C$ are submersions.
Furthermore,
\[
(p,f):Y\longrightarrow T\times C
\]
is a finite morphism of degree nine.
\end{lemma}

\begin{proof}
Near a point of $R$, use $\nu$ to choose the coordinate $t$ on $B_0$
from a coordinate on $C$. The first cover then has the form
\begin{equation}\label{ce335:eq:local-first}
q(z,t)=t,\quad a(z,t)=t+z^3,\quad R=(z=0).
\end{equation}
Thus $a$ is a submersion along $R$; away from $R$, the map $A$ is
\'etale and $a$ is again a submersion.

Because $\eta$ is \'etale, write the parameter point on $T$ locally as
$(x,t)$, with $\eta(x,t)=(x,t)$ in the corresponding coordinates on
$S$. Near the ramification of the second cover, its local form is
\begin{equation}\label{ce335:eq:local-second}
p(u,x,t)=(x,t),\qquad b(u,x,t)=(x+u^3,t).
\end{equation}
This shows that $b$ is a submersion at its ramification points.
Elsewhere, $\mathcal B$ is \'etale and the projection
$S\times_{B_0}T\to S$ is smooth. Hence $b$ is a submersion everywhere,
and so is $f=a\circ b$.

At a point of $W$ mapping into $R$, the unique point above the marked
point is the ramification point of the second cover. Combining
\eqref{ce335:eq:local-first} and \eqref{ce335:eq:local-second}, we obtain
\begin{equation}\label{ce335:eq:nested-model}
f(u,x,t)=t+(x+u^3)^3,\quad W=(x=0),
\quad f|_W(u,t)=t+u^9.
\end{equation}
In particular, $f|_W$ is a submersion there. At any other point of $W$,
the map $b$ is unramified in the $p$-fiber direction and its image is
outside $R$. The derivative of $a$ along the corresponding $q$-fiber
is then nonzero. Since the $p$-fiber direction lies in $T_W$, this
proves the remaining cases.

Finally, on each curve $Y_t$, the map $f$ is the composite
\[
Y_t\xrightarrow{\ b_t\ }S_{q(\eta(t))}
\xrightarrow{\ a_{q(\eta(t))}\ }C
\]
of two degree-three finite maps. Therefore every fiber of $(p,f)$ is
finite and its generic degree is nine. Properness implies that
$(p,f)$ is finite, and the same description shows that it is
surjective.
\end{proof}
The maps in the construction fit into the following
commutative diagram.
\[
\begin{tikzcd}[column sep=2.6em,row sep=4em]
Y
  \arrow[r,"\mathcal B"']
  \arrow[rr,bend left=14,"b"]
  \arrow[rrrr,bend left=20,"f=a\circ b"]
  \arrow[d,shift left=1.2ex,"p"]
&
S\times_{B_0}T
  \arrow[r,"\operatorname{pr}_S"']
  \arrow[dl,"\operatorname{pr}_T"']
&
S
  \arrow[r,"{A=(a,q)}"']
  \arrow[rr,bend left=14,"a"]
  \arrow[d,shift left=1.2ex,"q"]
&
C\times B_0
  \arrow[r,"\operatorname{pr}_C"']
  \arrow[dl,"\operatorname{pr}_{B_0}"']
&
C
\\
T
  \arrow[u,dashed,shift left=1.2ex,"\sigma_2"]
  \arrow[urr,dashed,"\eta"']
  \arrow[rr,"q\circ\eta"']
&&
B_0
  \arrow[u,dashed,shift left=1.2ex,"\sigma_1"]
  \arrow[urr,dashed,"\nu"']
\end{tikzcd}
\]

The solid arrows form a commutative diagram.
The dashed arrows are holomorphic maps recording
the ramification sections and their branch values.
Here $\sigma_1$ and $\sigma_2$ are the sections
determined by the unique ramification point in
each fiber of the first and second covers,
respectively. Their relations are
\[
\begin{aligned}
q\circ\sigma_1&=\operatorname{id}_{B_0},
&
a\circ\sigma_1&=\nu,&
p\circ\sigma_2&=\operatorname{id}_T,
&
b\circ\sigma_2&=\eta.
\end{aligned}
\]
Equivalently,
\[
A\circ\sigma_1=(\nu,\operatorname{id}_{B_0}),
\qquad
\mathcal B\circ\sigma_2=(\eta,\operatorname{id}_T).
\]

\subsubsection*{Connectedness and its persistence under parameter covers}

\begin{lemma}\label{ce335:lem:prime-pi1}
A ramified finite morphism of prime degree between smooth connected
projective curves induces a surjection on topological fundamental
groups.
\end{lemma}

\begin{proof}
A ramified covering of closed connected surfaces that
does not induce a surjection on fundamental groups
admits a factorization into two coverings of degrees
greater than one; see
\cite[pp.~1715 and~1719]{BedoyaGoncalves2010}.
Since degrees multiply under composition, such a
factorization is impossible when the original degree
is prime. The conclusion follows.
\end{proof}

\begin{lemma}\label{ce335:lem:connected-fibers}
The submersion $f:Y\to C$ has connected fibers. Moreover, if
$\tau:T'\to T$ is any connected finite \'etale cover and
\[
Y'=Y\times_TT',\qquad f'=f\circ\operatorname{pr}_Y,
\]
then $Y'$ is connected and $f':Y'\to C$ has connected fibers.
\end{lemma}

\begin{proof}
Let
\(
p':Y'=Y\times_T T'\to T',
\) \(
\pi:Y'\to Y
\)
be the natural projections. The morphism $p'$ is the base
change of $p$, so it is a proper holomorphic submersion
with connected curve fibers. Since $T'$ is connected,
$Y'$ is connected as well. 

The projection $\pi$ is finite \'etale, being the base
change of $\tau$. Thus
\(
f'=f\circ\pi:Y'\to C
\)
is again a proper holomorphic submersion.

We first show that $f'$ induces a surjection on
topological fundamental groups. Choose $t'\in T'$,
and set
\(
t=\tau(t'),\) \(
\beta=q(\eta(t)),\) \(
F=(p')^{-1}(t').
\)
Under the natural identification $F\simeq Y_t$, the
restriction of $f'$ to $F$ is the composite
\(
F\simeq Y_t
\xrightarrow{ b_t}
S_\beta
\xrightarrow{a_\beta }
C.
\)
Both $b_t$ and $a_\beta$ are ramified finite morphisms
of degree three. Applying
Lemma~\ref{ce335:lem:prime-pi1} to each map shows that
both induce surjections on fundamental groups.
Consequently,
\(
(f'|_F)_*:\pi_1(F)\to\pi_1(C)
\)
is surjective. If $j:F\hookrightarrow Y'$ denotes
the inclusion, then
\(
(f'|_F)_*=f'_*\circ j_*.
\)
It follows that
\(
f'_*:\pi_1(Y')\to\pi_1(C)
\)
is surjective. 

Now consider the Stein factorization
\[
Y'\xrightarrow{\ g\ }\widehat C
\xrightarrow{\ h\ }C,
\qquad f'=h\circ g.
\]
Here $g$ is surjective with connected fibers, $h$ is
finite, and $\widehat C$ is a connected normal
projective curve, hence smooth; see
\cite[Theorem~37.53.4]{Stacks}.

We claim that $h$ is unramified. If $h$ were ramified
at some $\widehat c\in\widehat C$, then
$dh_{\widehat c}=0$. For any
$y\in g^{-1}(\widehat c)$, the chain rule would give
\(
df'_y=dh_{\widehat c}\circ dg_y=0,
\)
contrary to the fact that $f'$ is a submersion.
Thus $h$ is a connected finite \'etale cover of $C$.

Since $f'_*=h_*\circ g_*$ is surjective, $h_*$ is
surjective as well. For a connected finite covering,
the degree equals the index of the image of its
fundamental group. Hence
\[
\deg h
=
\bigl[\pi_1(C):h_*\pi_1(\widehat C)\bigr]
=1.
\]
Therefore $h$ is an isomorphism. Since $g$ has
connected fibers, so does $f'$.

Taking $T'=T$ and $\tau=\operatorname{id}_T$ proves
that $f:Y\to C$ has connected fibers.
\end{proof}

\subsection{An explicit \'etale extraction of a fifth root}
\label{ce335:sec:root}

The divisor $D$ need not initially have a fifth root in $\operatorname{Pic}(T)$.
We next arrange such a root by a specified finite \'etale base change.
The construction concerns actual line bundles, not merely numerical
classes.

\begin{lemma}\label{ce335:lem:root-on-product}
Let $H_1,H_2$ be smooth connected projective curves of positive genus,
and let $L$ be a line bundle on $H_1\times H_2$. For $j=1,2$, let
\(
\lambda_j:H'_j\to H_j
\)
be the connected characteristic \'etale cover corresponding to
\(
\ker\bigl(\pi_1(H_j)\to H_1(H_j,\mathbb{Z}/5)\bigr).
\)
Put $\lambda=\lambda_1\times\lambda_2$. Then there is a line bundle
$M$ on $H'_1\times H'_2$ with
\(
M^{\otimes5}\simeq\lambda^*L.
\)
\end{lemma}

\begin{proof}
Write
\(
V=H_1\times H_2,\) \(
V'=H'_1\times H'_2,\) \(
\mathbb F_5=\mathbb Z/5\mathbb Z.
\)
We first show that
\[
\lambda^*:H^2(V,\mathbb F_5)
\longrightarrow H^2(V',\mathbb F_5)
\]
is the zero map.

For each $j=1,2$, recall the identification
\(
H^1(H_j,\mathbb F_5)
\simeq
\operatorname{Hom}\bigl(\pi_1(H_j),\mathbb F_5\bigr).
\)
Under this identification, pullback by $\lambda_j$
is the restriction of homomorphisms to the subgroup
\(
K_j=(\lambda_j)_*\pi_1(H'_j)
=
\ker\bigl(\pi_1(H_j)\to
H_1(H_j,\mathbb F_5)\bigr).
\)
Every homomorphism from $\pi_1(H_j)$ to $\mathbb F_5$
factors through $H_1(H_j,\mathbb F_5)$, and therefore
vanishes on $K_j$. Hence
\(
\lambda_j^*:H^1(H_j,\mathbb F_5)
\to H^1(H'_j,\mathbb F_5)
\)
is zero.

The same conclusion holds in degree two. Indeed,
if $g_j=g(H_j)\geq1$, then
\(
\deg\lambda_j=5^{2g_j}.
\)
For a finite covering of compact oriented surfaces,
pullback on top-degree cohomology sends the orientation
generator to the degree times the orientation generator.
Since $5$ divides $\deg\lambda_j$, we obtain
\(
\lambda_j^*:H^2(H_j,\mathbb F_5)
\to H^2(H'_j,\mathbb F_5)
,\,\lambda^*_j=0,
\)
in the sense that this pullback map is zero.

Now the K\"unneth formula gives
\[
\begin{aligned}
H^2(V,\mathbb F_5)\simeq{}&
H^2(H_1,\mathbb F_5)\otimes_{\mathbb F_5}
H^0(H_2,\mathbb F_5)\\
&\oplus
H^1(H_1,\mathbb F_5)\otimes_{\mathbb F_5}
H^1(H_2,\mathbb F_5)\\
&\oplus
H^0(H_1,\mathbb F_5)\otimes_{\mathbb F_5}
H^2(H_2,\mathbb F_5).
\end{aligned}
\]
Under this decomposition, $\lambda^*$ acts on each
summand as the tensor product of the corresponding
pullback maps. Each summand contains a factor in
degree one or two, on which the relevant pullback
map is zero. Thus $\lambda^*$ is zero on
$H^2(V,\mathbb F_5)$.

We next apply this vanishing to the first Chern class
of $L$. Let $\operatorname{red}_5$ denote reduction
of integral cohomology classes modulo $5$.
Naturality of the first Chern class gives
\(
\operatorname{red}_5\bigl(c_1(\lambda^*L)\bigr)
=
\lambda^*\bigl(\operatorname{red}_5(c_1(L))\bigr)
=0.
\)
The coefficient sequence
\(
0\to\mathbb Z
\xrightarrow{\times5}\mathbb Z
\to\mathbb F_5
\to 0
\)
induces the exact sequence
\[
H^2(V',\mathbb Z)
\xrightarrow{\ \times5\ }
H^2(V',\mathbb Z)
\xrightarrow{\ \operatorname{red}_5\ }
H^2(V',\mathbb F_5).
\]
Exactness therefore implies that
\[
c_1(\lambda^*L)=5\alpha
\qquad\text{for some }\alpha\in H^2(V',\mathbb Z).
\]

Since $\lambda^*L$ is holomorphic, its first Chern
class has type $(1,1)$. The image of $\alpha$ in
$H^2(V',\mathbb C)$ is one fifth of this class, so
it also has type $(1,1)$. The integral $(1,1)$
theorem gives a holomorphic line bundle $M_0$ on
$V'$ such that
\(
c_1(M_0)=\alpha;
\)
see \cite[Chapter~VI, \S9.1]{Dem12}.

To obtain the required fifth root, set
\(
P=\lambda^*L\otimes(M_0^\vee)^{\otimes5}.
\)
Then
\[
c_1(P)=c_1(\lambda^*L)-5c_1(M_0)=0.
\]
Thus $P$ belongs to $\operatorname{Pic}^0(V')$,
the complex torus of holomorphic line bundles with
zero integral first Chern class; see again
\cite[Chapter~VI, \S9.1]{Dem12}.
Multiplication by five on a complex torus is
surjective: a point represented by a vector $v$
has a preimage represented by $v/5$.
Consequently, there exists
$Q\in\operatorname{Pic}^0(V')$ such that
\(
Q^{\otimes5}\simeq P.
\)
Taking $M=M_0\otimes Q$, we conclude that
\(
M^{\otimes5}
\simeq M_0^{\otimes5}\otimes P
\simeq\lambda^*L,
\)
as required.
\end{proof}

Apply Lemma \ref{ce335:lem:root-on-product} to $H_1=C$, $H_2=B_0$, and
$L=\mathcal{O}_{C\times B_0}(\Gamma)$. Write the resulting product cover as
\(
\lambda:C'\times B'_0\to C\times B_0,
\) \(M^{\otimes5}\simeq\lambda^*\mathcal{O}(\Gamma).
\)
Set
\(
m=A\circ\eta:T\to C\times B_0.
\)
By \eqref{ce335:eq:triple-divisor},
\begin{equation}\label{ce335:eq:pullback-3D}
m^*\mathcal{O}(\Gamma)\simeq\mathcal{O}_T(3D).
\end{equation}
Choose any nonempty connected component $T_1$ of
\(
T\times_{C\times B_0}(C'\times B'_0).
\)
Let
\(
\tau:T_1\to T,
\) \( m_1:T_1\to C'\times B'_0
\)
be the induced maps, and put $D_1=\tau^{-1}(D)$.
The map $\tau$ is finite \'etale and surjective. The divisor $D_1$
is a nonempty disjoint union of smooth curves, all of negative
self-intersection.

For $N=m_1^*M$, \eqref{ce335:eq:pullback-3D} gives
\(
N^{\otimes5}\simeq\mathcal{O}_{T_1}(3D_1).
\)
Define
$
\mathscr{A}=N^{\otimes2}\otimes\mathcal{O}_{T_1}(-D_1).
$
The identity $2\cdot3-5=1$ gives an isomorphism of line bundles
\begin{equation}\label{ce335:eq:fifth-root}
\mathscr{A}^{\otimes5}\simeq\mathcal{O}_{T_1}(D_1).
\end{equation}

At the same time, make the corresponding \'etale change on the
threefold:
\begin{equation}\label{ce335:eq:Y1}
\begin{gathered}
Y_1=Y\times_TT_1,\,\, p_1:Y_1\to T_1,\,\,
f_1=f\circ\operatorname{pr}_Y,\,\,W_1=p_1^{-1}(D_1).
\end{gathered}
\end{equation}
The threefold $Y_1$ is smooth, connected, and projective. It inherits
negative holomorphic bisectional curvature from $Y$, since
$Y_1\to Y$ is \'etale. By Lemmas~\ref{ce335:lem:transversality} and~\ref{ce335:lem:connected-fibers},
$f_1$ is smooth with connected fibers, $f_1|_{W_1}$ is smooth, and
\begin{equation}\label{ce335:eq:finite-nine-after-basechange}
(p_1,f_1):Y_1\longrightarrow T_1\times C
\quad\text{is finite of degree nine.}
\end{equation}
% Required packages:
% \usepackage{amsmath,amssymb,mathrsfs}
% \usepackage{tikz-cd}

The maps and distinguished divisors in the construction
fit into the following commutative diagram. A pair $(Z,E)$
denotes a variety $Z$ together with a divisor $E\subset Z$;
all arrows denote morphisms of the underlying varieties.
\begin{equation}
\begin{tikzcd}[column sep=3.2em,row sep=3em]
& C &&
\\
T_1\times C
  \arrow[ur,"\operatorname{pr}_C"]
  \arrow[dr,"\operatorname{pr}_{T_1}"']
& (Y_1,W_1)
  \arrow[l,"{(p_1,f_1)}"']
  \arrow[u,"f_1"']
  \arrow[r,"\operatorname{pr}_Y"]
  \arrow[d,"p_1"']
& (Y,W)
  \arrow[ul,"f"']
  \arrow[d,"p"]
&
\\
& (T_1,D_1)
  \arrow[r,"\tau"]
  \arrow[d,"m_1"']
& (T,D)
  \arrow[r,"\eta"]
  \arrow[d,"m"']
& (S,R)
  \arrow[dl,"A"]
\\
& C'\times B'_0
  \arrow[r,"\lambda"']
& (C\times B_0,\Gamma)
&
\end{tikzcd}
\end{equation}

\subsection{The final cover and verification of the theorem}
\label{ce335:sec:final}

\subsubsection*{The cyclic cover and the smooth outer fibration}

Use \eqref{ce335:eq:fifth-root} and the canonical section defining $D_1$ to
form the cyclic cover
\begin{equation}\label{ce335:eq:cyclic-cover}
\sigma:\widehat T
=\operatorname{Spec}_{T_1}\left(\bigoplus_{j=0}^4\mathscr{A}^{-j}\right)
\longrightarrow T_1.
\end{equation}
The multiplication is determined by the section of $\mathscr{A}^{\otimes5}$
corresponding to $D_1$. This is the standard cyclic-cover algebra;
see \cite{Par91}. Locally its equation is $v^5=s_{D_1}$.
It is finite flat of degree five. Along $D_1$, choose a local defining
coordinate $x$; the equation is $v^5=x$, so $\widehat T$ is smooth.
Away from $D_1$ the cover is \'etale.

Since the defining section has nonempty reduced zero
divisor $D_1$, the cyclic cover $\widehat T$ is
irreducible, and hence connected; see
\cite[Remark~3.14(a) and Lemma~3.15(a)]
{EsnaultViehweg1992}.

Define the fiber product
\begin{equation}\label{ce335:eq:final-X}
X:=Y_1\times_{T_1}\widehat T,
\end{equation}
and let
\(
h:X\to Y_1,
\) \(
p':X\to\widehat T
\)
be the natural projections. Explicitly, a point of $X$ is a pair
$(y,\widehat t)$ satisfying $p_1(y)=\sigma(\widehat t)$, and
\(
h(y,\widehat t)=y
\) and \(
p'(y,\widehat t)=\widehat t.
\)
Consequently, these maps satisfy
\(
p_1\circ h=\sigma\circ p'.
\)
Set
\[
\pi=f_1\circ h:X\longrightarrow C.
\]
The map $p'$ is smooth with connected curve fibers of genus fourteen,
being a base change of $p_1$. Thus $X$ is smooth and connected. It is
projective, and $h$ is finite flat of degree five, branched exactly
along $W_1$.

We check that $\pi$ is smooth even at the ramification locus of $h$.
Since $f_1$ and $f_1|_{W_1}$ are submersions, there are local
coordinates $(x,y,t)$ on $Y_1$ near $W_1$ with
\[
f_1(x,y,t)=t,\qquad W_1=(x=0).
\]
The cover has coordinates $(v,y,t)$ with $x=v^5$, so $\pi(v,y,t)=t$.
Away from the ramification divisor, $h$ is \'etale and smoothness is
immediate. Hence $\pi$ is a smooth surjective projective morphism.

\subsubsection*{Connectedness and curvature of every surface fiber}

Fix $c\in C$ and set
\(
F_c=f_1^{-1}(c)\) and \(E_c=F_c\cap W_1.\)
The surface $F_c$ is smooth and connected. As a submanifold of $Y_1$,
it inherits a K\"ahler metric with strictly negative holomorphic
bisectional curvature.

Note that $E_c$ is the fiber over $c$ of the smooth map
$f_1|_{W_1}:W_1\to C$, so it is a smooth curve, possibly
with several disjoint connected components.
To see explicitly why the intersection is reduced, choose a local
coordinate on $C$ that vanishes at $c$. Since $f_1|_{W_1}$ is smooth,
near every point of $E_c$ we can choose holomorphic coordinates
$(x,y,t)$ on $Y_1$ such that
\(
f_1(x,y,t)=t\) and \( W_1=\{x=0\}.
\)
In these coordinates,
\(
F_c=\{t=0\}\) and \( E_c=\{x=t=0\}.
\)
Hence $(x,y)$ are local coordinates on $F_c$, and $E_c$ is defined
there by the single equation $x=0$. In particular, every component
occurs with multiplicity one. This description also applies at
points where the ramification points of the two covers coincide.

Moreover, \eqref{ce335:eq:finite-nine-after-basechange} implies that
$p_1|_{F_c}:F_c\to T_1$ is finite and surjective. Since $D_1$ is
nonempty, its inverse image
\(
E_c=(p_1|_{F_c})^{-1}(D_1)
\)
is nonempty as well. Finally, $W_1$ is the pullback of $D_1$ under
$p_1$, so restricting its defining equation to $F_c$ gives the
pullback divisor $(p_1|_{F_c})^*D_1$. The local calculation above
shows that this divisor is reduced. Therefore, as Cartier
divisors on $F_c$,
\begin{equation}\label{ce335:eq:Ec-pullback}
E_c=(p_1|_{F_c})^*D_1.
\end{equation}

Let $E$ be a connected component of $E_c$, mapping onto a component
$D_{1,i}$ of $D_1$. The map $E\to D_{1,i}$ is finite and surjective.
Because all components of both divisors are disjoint and
\eqref{ce335:eq:Ec-pullback} is reduced, restriction of their divisor line
bundles gives
\begin{equation}\label{ce335:eq:Ec-normal}
N_{E/F_c}\simeq(p_1|_E)^*N_{D_{1,i}/T_1}.
\end{equation}
Therefore
\begin{equation}\label{ce335:eq:Ec-degree}
E^2=\deg(E\to D_{1,i})\,D_{1,i}^2<0.
\end{equation}

The restriction
\[
h_c:X_c\longrightarrow F_c
\]
is a cyclic cover of degree five, totally ramified along $E_c$ and
unramified elsewhere. Its connectedness follows again from \cite[Remark~3.14(a) and Lemma~3.15(a)]
{EsnaultViehweg1992}. Smoothness has already been proved
by the submersion property of $\pi$.

Now Proposition~\ref{ce335:prop:cyclic-metric} applies to $h_c$, by
\eqref{ce335:eq:Ec-degree}. Consequently, every $X_c$ admits a K\"ahler
metric with strictly negative holomorphic bisectional curvature.
 The curve $C$ has genus two and admits
a hyperbolic metric, which supplies the required curvature on the
base.

\subsubsection*{A relative K\"ahler form with negatively curved fiber metrics}

Fix a K\"ahler form $\omega_{Y_1}$ on $Y_1$ with strictly negative
holomorphic bisectional curvature, and set
\(
\eta=h^*\omega_{Y_1}.
\)
The form $\eta$ is smooth, real, and $d$-closed on $X$. For every
$c\in C$, the metric on $X_c$ supplied by
Proposition~\ref{ce335:prop:cyclic-metric} is constructed from
$h_c^*(\omega_{Y_1}|_{F_c})=\eta|_{X_c}$. More precisely,
\eqref{ce335:eq:global-metric} expresses it as
\[
\eta|_{X_c}
 +\sqrt{-1}\partial_c\bar\partial_c\varphi_c,
\qquad
\varphi_c=\varepsilon_c\chi_c(\rho_c+K_c\rho_c^2)
 \in C^\infty(X_c,\mathbb R).
\]
The cutoff makes this a globally defined smooth function on $X_c$.
Thus all the fiber metrics arise from the restriction of the same
global closed form by adding fiber potentials. No smooth dependence
of the auxiliary choices $\varepsilon_c,\chi_c,\rho_c,K_c$ on $c$ is
needed.

Proposition~\ref{ce335:prop:relative-negative-form} now gives a real
smooth function $\Phi$ on $X$ such that
\begin{equation}\label{ce335:eq:final-relative-form}
\omega=h^*\omega_{Y_1}
       +\sqrt{-1}\partial\bar\partial\Phi
\end{equation}
is a relative K\"ahler form for $\pi:X\to C$, and
\begin{equation}\label{ce335:eq:final-fiber-negativity}
\operatorname{HBC}(\omega|_{X_c})<0
\qquad\text{for every }c\in C.
\end{equation}
In particular, $[\omega]=h^*[\omega_{Y_1}]$. If $\omega_C$ is a
hyperbolic K\"ahler form on $C$, then
\(
\omega_k=\omega+k\pi^*\omega_C
\)
is a K\"ahler form on $X$ for all sufficiently large $k$, and
$\omega_k|_{X_c}=\omega|_{X_c}$ for every $c$. Hence the relative form
in the statement of Theorem~\ref{ce335:thm:main} may also be chosen
positive definite on the total space.

There is a further consequence of this construction. The fiber
metrics in \eqref{ce335:eq:final-fiber-negativity} have strictly
negative holomorphic sectional curvature, and the same is true of
$\omega_C$. By \cite[Theorem~1.6]{Wan26}, $\omega_k$ has strictly
negative holomorphic sectional curvature for all sufficiently large
$k$. The obstruction below will show that $X$ admits no K\"ahler
metric with strictly negative holomorphic bisectional curvature.

\subsubsection*{Kodaira--Spencer injectivity at every parameter}

Since $Y_1$ has strictly negative holomorphic bisectional curvature,
Lemma~\ref{ce335:lem:negative-vanishing} gives
\(
H^0(F_c,T_{Y_1}|_{F_c})=0
\)
and so
$\rho_{f_1,c}$ is injective.
Apply Lemma~\ref{ce335:lem:trace-KS} to $h:X\to Y_1$ over $C$, one concludes that $\rho_{\pi,c}$ is injective for every $c\in C$.

\subsubsection*{The tangent-bundle obstruction on the threefold}

Recall that
\(
X=Y_1\times_{T_1}\widehat T,
\)
with projections $h:X\to Y_1$ and $p':X\to\widehat T$.
Choose a point $\widehat t$ on the ramification curve of
$\sigma:\widehat T\to T_1$. Since $\sigma$ is ramified at
$\widehat t$, we can choose a nonzero tangent vector
\(
\xi\in T_{\widehat t}\widehat T,
\) \(d\sigma_{\widehat t}(\xi)=0.
\)
Let
\(
G=(p')^{-1}(\widehat t).
\)
The projection $h$ identifies $G$ with the fiber
$p_1^{-1}(\sigma(\widehat t))$. In particular, $G$ is a
smooth connected compact curve of genus fourteen.

We now use the fixed vector $\xi$ to construct a nowhere-vanishing
holomorphic section of $T_X|_G$. A point of $G$ has the form
$(y,\widehat t)$, where $p_1(y)=\sigma(\widehat t)$. The
fiber-product description of $X$ gives
\begin{equation}\label{ce335:eq:tangent-fiberproduct}
T_{(y,\widehat t)}X
=
\bigl\{(v,w)\in T_yY_1\oplus T_{\widehat t}\widehat T:
dp_1(v)=d\sigma_{\widehat t}(w)\bigr\}.
\end{equation}
Because $d\sigma_{\widehat t}(\xi)=0$, the pair $(0,\xi)$
satisfies this condition at every point of $G$. Thus
\(
s(y,\widehat t)=(0,\xi)
\)
defines a holomorphic section of $T_X|_G$.
Moreover,
\(
dp'\bigl(s(y,\widehat t)\bigr)=\xi\ne0,
\)
so $s$ is nowhere zero. Its span therefore defines a trivial
holomorphic line subbundle of $T_X|_G$, with inclusion
\begin{equation}\label{ce335:eq:trivial-subbundle}
\mathcal O_G\hookrightarrow T_X|_G,
\quad 1\longmapsto s.
\end{equation}

Finally, $h$ is the projection onto the first factor, so
$dh(s)=0$. Since $\pi=f_1\circ h$, we also have
\(
d\pi(s)=df_1\bigl(dh(s)\bigr)=0.
\)
Hence $s$ takes values in the relative tangent bundle
$T_{X/C}=\ker(d\pi)$, and the inclusion above factors as
\[
\mathcal O_G\hookrightarrow T_{X/C}|_G
\hookrightarrow T_X|_G.
\]

By Lemma~\ref{ce335:lem:tangent-obstruction}, \eqref{ce335:eq:trivial-subbundle}
excludes every K\"ahler metric of strictly negative holomorphic
bisectional curvature on $X$. Dualizing it gives a quotient
\(
\Omega_X^1|_G\twoheadrightarrow\mathcal{O}_G,
\)
which also proves that $\Omega_X^1$ is not ample.

Finally, the restriction of $h$ identifies $G$ with the curve fiber
$(Y_1)_{\sigma(\widehat t)}$. The restriction of $f_1$ to that curve
is the composite of the two degree-three maps used in the tower.
Hence
\(
\deg(\pi|_G)=3\cdot3=9.
\)
In particular, $G$ is not contained in an outer fiber $X_c$.
The tangent-bundle obstruction therefore does not contradict the
negative curvature established on every $X_c$. This completes the proof of Theorem~\ref{ce335:thm:main}.

\subsection{Remarks on the construction}
\label{ce335:sec:remarks}

The notation $(3,3,5)$ records the degrees of the two covers of curves
and the final cyclic cover.  All these additional covers
are finite, and every space in the construction remains projective.
The bottom curve $C$ itself is never replaced: it is the fixed genus-two
curve \eqref{ce335:eq:base-curve}.

The relative form constructed in
\eqref{ce335:eq:final-relative-form} induces a smooth Hermitian
metric on $T_{X/C}$. Its restriction to each
$T_{X/C}|_{X_c}=T_{X_c}$ is Griffiths negative. This assertion
concerns curvature in directions tangent to the fiber. Griffiths
negativity on the whole relative tangent bundle would require
strict negativity in every nonzero direction of $T_X$. The trivial
line subbundle along $G$ excludes such a metric on $T_{X/C}$, so
the example does not satisfy the stronger hypothesis in
\cite[Theorem~1.2]{Wan26}. 

There is also a distinction between the intrinsic curvature of
the induced fiber metric and the curvature of an ambient
K\"ahler metric evaluated on vertical tangent vectors. The Gauss
equation relates them through the second fundamental form. Thus
the global K\"ahler forms $\omega_k$ constructed above can induce
strictly negatively curved metrics on all fibers without having
strictly negative bisectional curvature on $X$.

Likewise, the inner curve family $p':X\to\widehat T$ is not effectively
parametrized at a ramification point of $\sigma$: its
Kodaira--Spencer map annihilates $\ker d\sigma$. The everywhere
injective Kodaira--Spencer map in Theorem~\ref{ce335:thm:main} belongs to the
\emph{different} fibration $\pi:X\to C$, whose fibers are surfaces.

\section{Products and the generalized Weil--Petersson metric}
\label{sec:product-MA}

We now consider relative K\"ahler fibrations whose relative forms
satisfy the homogeneous Monge--Amp\`ere equation. The geometry of
these forms determines a natural metric on the base. The product
formula below identifies a source of equality in its bisectional
curvature estimate, complementing the obstruction on the total
space established in Section~\ref{ce335:sec:counterexample}.

\subsection{Monge--Amp\`ere forms and their Kodaira--Spencer tensors}

\begin{definition}
Let $p:(X,\omega)\to B$ be a relative K\"ahler fibration of
relative dimension $n>0$. The relative form $\omega$ is called
a \emph{Monge--Amp\`ere form} if
\begin{equation}\label{eq:MA-definition}
 \omega^{n+1}=0.
\end{equation}
The pair $p:(X,\omega)\to B$ is then called a
\emph{Monge--Amp\`ere fibration}.
\end{definition}

These conditions imply that $\omega$ is semipositive of constant
complex rank $n$. Indeed, in adapted coordinates its vertical
Hermitian block is positive definite. Eliminating the mixed blocks
by a change of frame leaves this positive block and its horizontal
Schur complement. Equation~\eqref{eq:MA-definition} forces the
Schur complement to vanish. Thus the null space of $\omega$ is a
smooth complex distribution complementary to the vertical tangent
bundle. It is the $\omega$-horizontal distribution. We write
$\omega_t=\omega|_{X_t}$ for the induced fiber metric.

For clarity, we distinguish the classical Kodaira--Spencer class
$\rho_{p,t}(v)$ from its representative determined by $\omega$.
If $v$ is a local holomorphic vector field on $B$, let $V$ be its
unique $(1,0)$ horizontal lift, characterized by
\(
 dp(V)=v\) and \( \omega(V,\overline U)=0\)
 for every vertical $(1,0)$  vector $U$.

The tensor
\begin{equation}\label{eq:omega-KS-definition}
 \kappa_v=(\bar\partial V)|_{X_t}
 \in A^{0,1}(X_t,T^{1,0}X_t)
\end{equation}
is $\bar\partial$-closed and represents $\rho_{p,t}(v)$.
It need not be the harmonic representative. In local base
coordinates $t^j$ and fiber coordinates $z^\alpha$, writing the
vertical block of $\omega$ as $(g_{\alpha\bar\beta})$, the lift is
\[
 V_j=\frac{\partial}{\partial t^j}
 -g_{j\bar\beta}g^{\bar\beta\alpha}
 \frac{\partial}{\partial z^\alpha}.
\]
This also shows directly that the construction depends smoothly on
the parameter and is intrinsic.

Following \cite[Definition~1.12]{WanWang2023}, define
\begin{equation}\label{eq:generalized-WP-definition}
 G_{\mathrm{WP},t}(v,\bar w)=
 \int_{X_t}\langle\kappa_v,\kappa_w\rangle_{\omega_t}
 \frac{\omega_t^n}{n!}.
\end{equation}
Here the pointwise pairing is the Hermitian pairing on
$\operatorname{Hom}(T^{0,1}X_t,T^{1,0}X_t)$ induced by $\omega_t$.
If $\rho_{p,t}$ is injective, then every nonzero $v$ has a nonzero
representative $\kappa_v$, so the integral in
\eqref{eq:generalized-WP-definition} is positive. The metric is
therefore positive definite for every effectively parametrized
family. We use the unnormalized integral; dividing by the fiber
volume would change the constant factors in the product formula.

The relevant curvature theorem takes the following form.

\begin{theorem}\label{thm:WW-curvature}
For an effectively parametrized Monge--Amp\`ere fibration, the metric
$G_{\mathrm{WP}}$ is K\"ahler and satisfies
\begin{equation}\label{eq:WW-curvature-bound}
 R^{\mathrm{WP}}(\xi,\bar\xi,\eta,\bar\eta)
 \leq -\frac{2}{n\Vol(X_t)}
       |G_{\mathrm{WP}}(\xi,\bar\eta)|^2,
 \qquad
 \Vol(X_t)=\int_{X_t}\frac{\omega_t^n}{n!}.
\end{equation}
In particular its holomorphic bisectional curvature is nonpositive,
and
\[
 \HSC(G_{\mathrm{WP}})\leq-\frac{2}{n\Vol(X_t)}<0.
\]
\end{theorem}

This is \cite[Theorem~2.4, equation~(2.21), and Corollary~2.6]{WanWang2023},
with the convention~\eqref{ce335:eq:curvature-convention}.
The fiber volume is locally constant because $\omega$ is closed and
fiber integration commutes with exterior differentiation. The curvature is strictly negative whenever
$G_{\mathrm{WP}}(\xi,\bar\eta)\ne0$.
For orthogonal tangent vectors, the estimate permits
the curvature to vanish.

\subsection{Products of Monge--Amp\`ere fibrations}
\label{subsec:product-MA}

Let
\(
 p_i:(X_i,\omega_i)\to B_i,
\, i=1,2,
\)
be Monge--Amp\`ere fibrations, and let \(n_i\) be the complex dimension of the
fibers of \(p_i\). Put
\[
 X:=X_1\times X_2,
 \qquad
 B:=B_1\times B_2,
 \qquad
 p:=p_1\times p_2,
\]
and denote the projections from \(X\) by \(q_i:X\to X_i\). Define
\begin{equation}
 \omega:=q_1^*\omega_1+q_2^*\omega_2.
 \label{eq:product-MA-form}
\end{equation}

\begin{proposition}
\label{prop:product-stability-MA}
The map
\(
 p:(X,\omega)\to B
\)
is a Monge--Amp\`ere fibration whose fibers have complex dimension
\(n=n_1+n_2\). If both \(X_1\) and \(X_2\) are compact, then \(X\) is
compact.
\end{proposition}

\begin{proof}
The map \(p\) is a proper holomorphic submersion because it is the product
of two proper holomorphic submersions. Its fiber over \(t=(t_1,t_2)\) is
\(
 X_t=X_{1,t_1}\times X_{2,t_2}.
\)
The form \(\omega\) is real, smooth, of type \((1,1)\), and \(d\)-closed. Its
restriction to \(X_t\) is
\(
 \omega_t
 =q_{1,t}^*\omega_{1,t_1}+q_{2,t}^*\omega_{2,t_2},
\)
where \(q_{i,t}:X_t\to X_{i,t_i}\) is the projection. If
\(u=(u_1,u_2)\in T_xX_t\) is nonzero, then
\[
 \omega_t(u,Ju)
 =
 \omega_{1,t_1}(u_1,J_1u_1)
 +
 \omega_{2,t_2}(u_2,J_2u_2)>0.
\]
Thus \(\omega_t\) is a K\"ahler form.

It remains to verify the homogeneous Monge--Amp\`ere equation. Since
\(
 \omega_1^{n_1+1}=0,
\) \(
 \omega_2^{n_2+1}=0,
\)
the binomial expansion gives
\[
 \omega^{n+1}
 =
 \sum_{a=0}^{n+1}
 \binom{n+1}{a}
 q_1^*\omega_1^a\wedge q_2^*\omega_2^{n+1-a}.
\]
For every index \(a\), either \(a\ge n_1+1\), in which case
\(\omega_1^a=0\), or \(a\le n_1\), in which case
\(
 n+1-a=n_1+n_2+1-a\ge n_2+1
\)
and hence \(\omega_2^{n+1-a}=0\). Every term therefore vanishes, so
\(
 \omega^{n+1}=0.
\)
This proves that \(p:(X,\omega)\to B\) is Monge--Amp\`ere. The final
assertion follows from compactness of a finite product of compact spaces.
\end{proof}

We next describe the horizontal lifts and Kodaira--Spencer tensors of the
product. Let \(v_i\) be a vector field on \(B_i\), and let \(V_i\) be its
\(\omega_i\)-horizontal lift to \(X_i\). Since the two summands in
\eqref{eq:product-MA-form} have no mixed terms, the \(\omega\)-horizontal
lift of \((v_1,v_2)\) is
\begin{equation}
 V_{(v_1,v_2)}=(V_1,V_2).
 \label{eq:product-horizontal-lift}
\end{equation}
Consequently, under the natural splitting
\[
 T^{1,0}X_t
 =
 q_{1,t}^*T^{1,0}X_{1,t_1}
 \oplus
 q_{2,t}^*T^{1,0}X_{2,t_2},
\]
the \(\omega\)-Kodaira--Spencer tensor is block diagonal:
\begin{equation}
 \kappa_{(v_1,v_2)}
 =
 \begin{pmatrix}
  q_{1,t}^*\kappa^{(1)}_{v_1} & 0\\
  0 & q_{2,t}^*\kappa^{(2)}_{v_2}
 \end{pmatrix}.
 \label{eq:product-KS-tensor}
\end{equation}
Here \(\kappa^{(i)}_{v_i}\) denotes the
\(\omega_i\)-Kodaira--Spencer tensor of \(p_i\) in the direction \(v_i\).

\begin{proposition}
\label{prop:product-KS-injective}
Assume that the Kodaira--Spencer maps of \(p_1\) and \(p_2\) are injective
at every point. Then the Kodaira--Spencer map of \(p\) is injective at every
point.
\end{proposition}

\begin{proof}
At \(t=(t_1,t_2)\), the Kodaira--Spencer class of \((v_1,v_2)\) is
represented by the block-diagonal tensor in
\eqref{eq:product-KS-tensor}. Equivalently, under the natural K\"unneth
inclusions, one has
\[
 \rho_t(v_1,v_2)
 =
 q_{1,t}^*\rho_{1,t_1}(v_1)
 +
 q_{2,t}^*\rho_{2,t_2}(v_2).
\]

Suppose that \(\rho_t(v_1,v_2)=0\). Restrict a representative to a slice
\(
 X_{1,t_1}\times\{x_2\}\subset X_t.
\)
Then project the restricted tangent bundle holomorphically onto
\(T X_{1,t_1}\). The second summand vanishes on the slice, since its form
component comes from the second factor. The first gives
\(\rho_{1,t_1}(v_1)\). Hence
\(
 \rho_{1,t_1}(v_1)=0.
\)
Injectivity of \(\rho_{1,t_1}\) gives \(v_1=0\). Restricting instead to a
slice
\(
 \{x_1\}\times X_{2,t_2}
\)
similarly gives \(v_2=0\). Therefore \(\rho_t\) is injective.
\end{proof}

Let
\[
 \operatorname{Vol}_i
 :=
 \int_{X_{i,t_i}}\frac{\omega_{i,t_i}^{n_i}}{n_i!}.
\]
Because \(\omega_i\) is \(d\)-closed and \(p_i\) is a proper smooth
fibration, \(\operatorname{Vol}_i\) is constant on every connected component
of \(B_i\). The fiberwise volume form of the product satisfies
\begin{equation}
 \frac{\omega_t^{n_1+n_2}}{(n_1+n_2)!}
 =
 q_{1,t}^*\frac{\omega_{1,t_1}^{n_1}}{n_1!}
 \wedge
 q_{2,t}^*\frac{\omega_{2,t_2}^{n_2}}{n_2!}.
 \label{eq:product-volume-form}
\end{equation}
Indeed, in the binomial expansion of \(\omega_t^{n_1+n_2}\), the only
nonzero term is the term of bidegree \((n_1,n_2)\) with respect to the two
factors.

\begin{proposition}
\label{prop:product-WP-splitting}
Assume the Kodaira--Spencer maps of the two factors are injective.
Let \(G_{\mathrm{WP},i}\) be the generalized Weil--Petersson metric of
\((X_i,\omega_i)\to B_i\), and let \(G_{\mathrm{WP}}\) be that of the
product. Then
\begin{equation}
 G_{\mathrm{WP}}
 =
 \operatorname{Vol}_2\,\operatorname{pr}_1^*G_{\mathrm{WP},1}
 +
 \operatorname{Vol}_1\,\operatorname{pr}_2^*G_{\mathrm{WP},2}.
 \label{eq:product-WP-metric}
\end{equation}
In particular, \(G_{\mathrm{WP}}\) is a Riemannian product metric, up to
constant positive rescalings of the two factors.
\end{proposition}

\begin{proof}
For \(v=(v_1,v_2)\) and \(w=(w_1,w_2)\), the orthogonal splitting of the
vertical tangent bundle and \eqref{eq:product-KS-tensor} give the pointwise
identity
\[
 \langle\kappa_v,\kappa_w\rangle_{\omega_t}
 =
 q_{1,t}^*
 \langle
  \kappa^{(1)}_{v_1},
  \kappa^{(1)}_{w_1}
 \rangle_{\omega_{1,t_1}}
 +
 q_{2,t}^*
 \langle
  \kappa^{(2)}_{v_2},
  \kappa^{(2)}_{w_2}
 \rangle_{\omega_{2,t_2}}.
\]
Integrating this equality against \eqref{eq:product-volume-form} and using
Fubini's theorem yields
\[
 \langle v,w\rangle_{\mathrm{WP}}
 =
 \operatorname{Vol}_2
 \langle v_1,w_1\rangle_{\mathrm{WP},1}
 +
 \operatorname{Vol}_1
 \langle v_2,w_2\rangle_{\mathrm{WP},2},
\]
which is precisely \eqref{eq:product-WP-metric}.
\end{proof}

\begin{corollary}
\label{cor:mixed-bisectional-zero}
Under the hypotheses of Proposition~\ref{prop:product-WP-splitting},
let \(\xi_1\in T_{t_1}B_1\) and
\(\eta_2\in T_{t_2}B_2\) be nonzero, and set
\(
 \xi=(\xi_1,0),
\) \(
 \eta=(0,\eta_2)
 \in T_{(t_1,t_2)}(B_1\times B_2).
\)
Then
\begin{equation}
 R^{\mathrm{WP}}(\xi,\overline\xi,\eta,\overline\eta)=0.
 \label{eq:mixed-bisectional-zero}
\end{equation}
\end{corollary}

\begin{proof}
Equation \eqref{eq:product-WP-metric} shows directly that
\(G_{\mathrm{WP}}\) is a product metric, so every mixed curvature component
vanishes.
\end{proof}

\section{Compact Monge--Amp\`ere fibrations from Shimura curves}
\label{sec:compact-Shimura}

We now give a compact algebraic realization of the product construction.
The starting point is a universal family of abelian surfaces over a compact
quaternionic Shimura curve. Its polarization defines a relative
K\"ahler form satisfying the Monge--Amp\`ere equation, while the
variation of the period lattice gives an everywhere nonzero
Kodaira--Spencer class. 

\subsection{The quaternionic family}
\label{subsec:compact-Shimura-family}

We construct a family of principally polarized abelian surfaces
over a compact Shimura curve. We first specify the quaternionic
data and then record the analytic description used below.

Let $D$ be an indefinite quaternion division algebra over
$\mathbb Q$, with maximal order $\mathcal O_D$ and discriminant
$d_D>1$. Thus $D$ is four-dimensional over $\mathbb Q$ and admits
a real splitting
\[
\sigma:D\otimes_{\mathbb Q}\mathbb R
\xrightarrow{\sim}M_2(\mathbb R).
\]
Here $\mathcal O_D$ is a subring which is a full integral lattice
in $D$, maximal among such subrings, and $d_D$ is the product of
the finite primes at which $D$ ramifies.
One may take $d_D=6$; see \cite[\S43.2]{Voight2021}.
Write $\beta^\iota$ for the main involution and
\[
\operatorname{trd}(\beta)=\beta+\beta^\iota,
\qquad
q_D(\beta)=\beta\beta^\iota
\]
for the reduced trace and norm. Under $\sigma$, these become the
matrix trace and determinant.

By \cite[43.6.6]{Voight2021}, there exists
$\mu\in\mathcal O_D$ such that
\begin{equation}\label{eq:mu-square}
\mu^2=-d_D,
\qquad
\beta^*=\mu^{-1}\beta^\iota\mu.
\end{equation}
The second formula defines a positive involution on $D$.
We choose the sign of $\mu$ as follows. For
$\tau\in\mathbb H:=\{\tau\in\mathbb C:\operatorname{Im}\tau>0\}$,
put
\begin{equation}\label{eq:QM-fiber-lattice}
v_\tau=\begin{pmatrix}\tau\\1\end{pmatrix},
\qquad
\Lambda_\tau=\sigma(\mathcal O_D)v_\tau,
\qquad
A_\tau=\mathbb C^2/\Lambda_\tau.
\end{equation}
The lattice $\Lambda_\tau$ has rank four. After replacing $\mu$
by $-\mu$ if necessary, the form
\[
E_{\mu,\tau}
\bigl(\sigma(\alpha)v_\tau,\sigma(\beta)v_\tau\bigr)
=
\frac{1}{d_D}\operatorname{trd}(\mu\alpha\beta^\iota),
\qquad \alpha,\beta\in\mathcal O_D,
\]
is an integral Riemann form defining a principal polarization
on every $A_\tau$; see
\cite[Lemmas~43.6.7, 43.6.16, and~43.6.22]{Voight2021}.
In particular, the $A_\tau$ are abelian surfaces.

Fix $N\ge3$ with $\gcd(N,d_D)=1$. We use the moduli problem
of triples $(A\to S,\iota_A,\eta)$ over complex schemes $S$,
where $A\to S$ is an abelian scheme of relative dimension two,
\(
\iota_A:\mathcal O_D\to\operatorname{End}_S(A)
\)
is a unital action satisfying
\[
\det\bigl(
d\iota_A(\beta)|_{\operatorname{Lie}(A/S)}
\bigr)
=q_D(\beta)\cdot1_{\mathcal O_S},
\qquad \beta\in\mathcal O_D,
\]
and
\(
\eta:(\mathcal O_D/N\mathcal O_D)_S
\xrightarrow{\sim}A[N]
\)
is an $\mathcal O_D$-linear isomorphism.
Here $\operatorname{Lie}(A/S)$ is the rank-two tangent bundle
along the identity section, $A[N]=\ker[N]$, and the subscript
$S$ denotes the constant group scheme. Thus $\eta$ labels all
$N$-torsion points compatibly with the quaternionic action.
Isomorphisms of triples preserve both the action and the labeling.

This is the quaternionic PEL moduli problem with full level $N$.
The action and the fixed involution $*$ determine a unique
compatible principal polarization
$\lambda_A:A\xrightarrow{\sim}A^\vee$, characterized by
\(
\iota_A(\beta)^\vee\circ\lambda_A
=
\lambda_A\circ\iota_A(\beta^*).
\)
Here $A^\vee$ is the dual abelian scheme.
The polarization therefore need not be specified separately.
Take the principal congruence level $U=U(N)$, with $N\ge3$.
The quaternionic moduli stack described in
\cite[\S2.1, pp.~5--6]{Yuan2024} is then a scheme.
Let $M_N$ denote its base change to $\mathbb C$.
Since $D$ is a division algebra, $M_N$ is a smooth projective
curve, possibly disconnected; see
\cite[\S2.1--\S2.2, p.~6]{Yuan2024}.
The universal abelian scheme introduced in
\cite[\S1, p.~2]{Yuan2024} gives, after the same base change,
a family
\(
\mathcal U_N\to M_N.
\)
By the moduli interpretation, every family of the specified
triples over a complex scheme $S$ is obtained by pulling back
this family along its classifying morphism $S\to M_N$.

The surfaces in \eqref{eq:QM-fiber-lattice} carry the action
and level structure
\[
\iota_\tau(\beta)[z]=[\sigma(\beta)z],
\qquad
\eta_\tau(\alpha\bmod N\mathcal O_D)
=
\left[\tfrac{\sigma(\alpha)v_\tau}{N}\right].
\]
Let $C$ be the connected component of $M_N$ containing
$(A_{\mathrm i},\iota_{\mathrm i},\eta_{\mathrm i})$, and set
\begin{equation}\label{eq:universal-QM-surface}
\pi:\mathcal A:=\mathcal U_N\times_{M_N}C
\longrightarrow C.
\end{equation}
This is a smooth proper family of principally polarized
abelian surfaces. Specializing the adelic uniformization in
\cite[\S2.1, p.~5]{Yuan2024} to the connected component
chosen above gives
\begin{equation}\label{eq:Shimura-curve-uniformization}
C^{\mathrm{an}}\simeq\Gamma\backslash\mathbb H,
\quad
\Gamma=
\left\{
\gamma\in\mathcal O_D^\times:
q_D(\gamma)=1,\ 
\gamma\equiv1\pmod{N\mathcal O_D}
\right\}.
\end{equation}
Indeed, this component corresponds to the identity finite-adelic
class, whose stabilizer is $D_+^\times\cap U(N)=\Gamma$.
The group $\Gamma$ is torsion-free for $N\ge3$.
To see this, write $\gamma=1+N\alpha$ with
$\alpha\in\mathcal O_D$. The identity $q_D(\gamma)=1$ gives
\[
\operatorname{trd}(\gamma)
=2-N^2q_D(\alpha)\in2+N^2\mathbb Z.
\]
If $\gamma$ has finite order, its trace lies in $[-2,2]$,
so it must equal $2$. A finite-order matrix with determinant
one and trace two is the identity.

Let
\(
\mathcal O_D^1
:=\{\gamma\in\mathcal O_D^\times:q_D(\gamma)=1\}.
\)
The subgroup $\Gamma$ has finite index in $\mathcal O_D^1$,
since it is the kernel of reduction modulo $N\mathcal O_D$.
Because $D$ is a division algebra, the image of
$\mathcal O_D^1$ in $\operatorname{PSL}_2(\mathbb R)$ is
cocompact; see \cite[Main Theorem~38.4.3]{Voight2021}.
The same holds for its finite-index subgroup given by
the image of $\Gamma$. Thus
$C^{\mathrm{an}}\simeq\Gamma\backslash\mathbb H$ is compact.

Moreover, since $\Gamma$ is discrete and torsion-free,
its action on $\mathbb H$ is free. The quotient map
$\mathbb H\to C^{\mathrm{an}}$ is therefore an unramified
covering, and the hyperbolic metric on $\mathbb H$
descends to a smooth metric of Gaussian curvature $-1$
on $C$. 
Consequently, $g(C)\ge2$.

The total space $\mathcal A$ is projective as well. Indeed, if
$\mathcal P$ is the normalized Poincar\'e bundle on
$\mathcal A\times_C\mathcal A^\vee$, then
\(
(\operatorname{id}_{\mathcal A},\lambda_{\mathcal A})^*\mathcal P
\)
is relatively ample by
\cite[Proposition~1.3.2.18]{Lan2008}.
Since $C$ is projective, $\mathcal A$ is a smooth projective
threefold.

Finally, the pullback of \eqref{eq:universal-QM-surface} to
$\mathbb H$ has the explicit description
\begin{equation}\label{eq:analytic-universal-QM-family}
\mathcal A_{\mathbb H}
=
\mathcal O_D\backslash(\mathbb H\times\mathbb C^2),
\qquad
\beta\cdot(\tau,z)
=
\bigl(\tau,z+\sigma(\beta)v_\tau\bigr);
\end{equation}
see \cite[\S4.2]{Yuan2024}.
Its fiber over $\tau$ is precisely $A_\tau$.
These abelian surfaces with quaternionic multiplication are also
called \emph{false elliptic curves}; see \cite[\S1]{Buzzard1997}.

\subsection{The Betti form}
\label{subsec:Betti-MA-form}

We recall the Betti form of a polarized abelian scheme and verify its
Monge--Amp\`ere property. The construction is intrinsic to the polarized
integral local system; see
\cite[Proposition~2.2 and \S2.2]{DimitrovGaoHabegger2021}.

On a sufficiently small simply connected open subset $U\subset C$,
choose a flat symplectic basis of $R_1\pi_*\mathbb Z$. Here $R_1\pi_*\mathbb Z$ denotes the integral homology local
system whose fiber at $t\in C$ is $H_1(A_t,\mathbb Z)$. In normalized
period coordinates the family is
\begin{equation}
 \mathcal A|_U
 =\bigl(U\times\mathbb C^2\bigr)
   /\bigl(\mathbb Z^2+Z(t)\mathbb Z^2\bigr),
 \quad Z:U\longrightarrow\mathfrak H_2,
 \label{eq:local-period-description}
\end{equation}
where $Z$ is holomorphic and symmetric and $Y=\operatorname{Im}Z$ is
positive definite, and $\mathfrak{H}_2:=\{Z \in M_2(\mathbb{C}): Z^{\top}=Z, \operatorname{Im} Z>0\}$. Write
\begin{equation}
 w=x+Z(t)y.
 \label{eq:Betti-coordinates}
\end{equation}
Here $(x,y)$ is a point of $\mathbb R^4/\mathbb Z^4$; the individual
real functions $x_\alpha,y_\alpha$ are defined only after choosing local
lifts. Their differentials, and the form
\begin{equation}
 \omega_{\mathrm B}
 =2\sum_{\alpha=1}^2dx_\alpha\wedge dy_\alpha,
 \label{eq:Betti-form-real-coordinates}
\end{equation}
are well-defined on $\mathcal A|_U$. 

\begin{lemma}
\label{lem:Betti-form-properties}
The forms \eqref{eq:Betti-form-real-coordinates} define a global smooth
real closed $(1,1)$-form on $\mathcal A$. With
$
 \delta w=dw-(dZ)y=dx+Zdy,
$
one has
\begin{equation}
 \omega_{\mathrm B}
 =\sqrt{-1}\sum_{\alpha,\beta=1}^2
 (Y^{-1})_{\alpha\beta}\,
 \delta w_\alpha\wedge\overline{\delta w_\beta}.
 \label{eq:Betti-form-complex-expression}
\end{equation}
In particular, $\omega_{\mathrm B}$ is semipositive and has constant
complex rank $2$.
\end{lemma}

\begin{proof}
Changing a local lift of $(x,y)$ adds an integral constant vector.
Changing the flat symplectic basis acts by a constant integral symplectic
matrix. Both transformations preserve
$2\sum_\alpha dx_\alpha\wedge dy_\alpha$. The local forms therefore
glue, and their closedness follows from the same formula.

Under a lattice translation $w\mapsto w+m+Zn$, one has $y\mapsto y+n$,
and hence
\[
 d(w+m+Zn)-(dZ)(y+n)=dw-(dZ)y.
\]
Thus $\delta w$ descends under the lattice action. Since $Z$ is holomorphic, each component of
$\delta w=dw-(dZ)y$ is a smooth $(1,0)$-form.
To compute their pairing, write $Z=X+\sqrt{-1}Y$. The matrices $Y^{-1}$
and
\(
 ZY^{-1}\overline Z=XY^{-1}X+Y
\)
are symmetric. The $dx\wedge dx$ and $dy\wedge dy$ terms in
$(\delta w)^{\mathsf T}Y^{-1}\wedge\overline{\delta w}$ consequently
vanish, while the other two terms give
\[
 (dx)^{\mathsf T}Y^{-1}(\overline Z-Z)\wedge dy
 =-2\sqrt{-1}(dx)^{\mathsf T}\wedge dy.
\]
This proves \eqref{eq:Betti-form-complex-expression}.

The map $T^{1,0}\mathcal A\to\mathbb C^2$ given locally by
$v\mapsto\delta w(v)$ restricts to an isomorphism on the vertical
tangent space. It is therefore surjective. Since $Y^{-1}$ is positive
definite, \eqref{eq:Betti-form-complex-expression} proves both
semipositivity and constant complex rank $2$.
\end{proof}

The local form also admits the potential
\begin{equation}
 \omega_{\mathrm B}
 =\sqrt{-1}\,\partial\bar\partial
 \left(2(\operatorname{Im}w)^{\mathsf T}
 Y^{-1}(\operatorname{Im}w)\right).
 \label{eq:Betti-form-potential}
\end{equation}
For example, put $v=\operatorname{Im}w$ and $y=Y^{-1}v$.
Differentiation gives
\[
 \partial y=\frac{1}{2\sqrt{-1}}Y^{-1}\delta w,
 \qquad
 \bar\partial y=-\frac{1}{2\sqrt{-1}}Y^{-1}\overline{\delta w}.
\]
For $F=2v^{\mathsf T}Y^{-1}v$, one has
$\partial F=-2\sqrt{-1}y^{\mathsf T}dw
+\sqrt{-1}y^{\mathsf T}(dZ)y$.
Substituting the displayed derivatives in
$\partial\bar\partial F=-\bar\partial\partial F$ recovers
\eqref{eq:Betti-form-complex-expression}, see \cite[Lemma 2.3]{DimitrovGaoHabegger2021}.

\begin{proposition}
\label{prop:Betti-form-MA}
The polarized abelian scheme \eqref{eq:universal-QM-surface}, equipped
with $\omega_{\mathrm B}$, is a Monge--Amp\`ere fibration. More precisely,
\begin{equation}
 \left.\omega_{\mathrm B}\right|_{\mathcal A_t}
 =\sqrt{-1}\sum_{\alpha,\beta=1}^2
 (Y(t)^{-1})_{\alpha\beta}
 dw_\alpha\wedge d\overline w_\beta>0,
 \quad \omega_{\mathrm B}^{3}=0.
 \label{eq:Betti-MA-equation}
\end{equation}
\end{proposition}

\begin{proof}
On a fiber, $dZ$ vanishes and $\delta w=dw$, so the first assertion
follows from the positivity of $Y(t)$. The second follows from the
constant-rank statement in Lemma~\ref{lem:Betti-form-properties}.
Together with closedness, these are exactly the required properties.
\end{proof}

The kernel of $\omega_{\mathrm B}$ is the Betti horizontal distribution.
Locally its leaves are given by holding $(x,y)$ constant in
\eqref{eq:Betti-coordinates}. They are holomorphic graphs because $Z$
is holomorphic. The form is thus degenerate in precisely one complex
direction on $\mathcal A$, although it is positive in every nonzero
vertical direction.

\subsection{The Kodaira--Spencer map}
\label{subsec:Shimura-KS-injective}

We verify infinitesimal variation directly on the uniformization
\eqref{eq:analytic-universal-QM-family}. This also shows that the
$\omega_{\mathrm B}$-Kodaira--Spencer representatives are harmonic.

\begin{proposition}
\label{prop:Shimura-KS-injective}
For every $t\in C$, the classical Kodaira--Spencer map
\[
 \rho_t:T_tC\longrightarrow H^1(\mathcal A_t,T_{\mathcal A_t})
\]
is injective.
\end{proposition}

\begin{proof}
Write $\tau=r+\mathrm i s$, with $s>0$.
For each fixed $\tau$, every $z\in\mathbb C^2$ can be
written uniquely as
\[
z=M
\begin{pmatrix}\tau\\1\end{pmatrix},
\qquad
M=
\begin{pmatrix}
a_1&b_1\\
a_2&b_2
\end{pmatrix}
\in M_2(\mathbb R).
\]
Indeed, the equations $z_\alpha=a_\alpha\tau+b_\alpha$
give
\[
a_\alpha=\frac{\operatorname{Im}z_\alpha}{s},
\qquad
b_\alpha=\operatorname{Re}z_\alpha-r a_\alpha,
\qquad \alpha=1,2.
\]
Replacing $z$ by
$z+\sigma(\beta)(\tau,1)^{\mathsf T}$, with
$\beta\in\mathcal O_D$, replaces $M$ by $M+\sigma(\beta)$.
Thus the class
\(
[M]\in M_2(\mathbb R)/\sigma(\mathcal O_D)
\)
gives the Betti coordinates of the corresponding point
of the fiber.

Keeping these coordinates fixed means keeping $M$ fixed.
The resulting local section is represented by
\(
z_\alpha(\tau)=a_\alpha\tau+b_\alpha,
\)
and hence
\[
\frac{d z_\alpha(\tau)}{d\tau}
=a_\alpha
=\frac{\operatorname{Im}z_\alpha}{s}.
\]
Its tangent vector is the Betti horizontal lift of
$\partial/\partial\tau$:
\begin{equation}
V=\frac{\partial}{\partial\tau}
+\sum_{\alpha=1}^2
\frac{\operatorname{Im}z_\alpha}{s}
\frac{\partial}{\partial z_\alpha}.
\label{eq:QM-Betti-horizontal-lift}
\end{equation}

We check that this formula is compatible with the lattice
identifications. Write $\sigma(\beta)=(p\ \ q)$, where
$p,q\in\mathbb R^2$ are its columns. The corresponding
translation is
\(
T_\beta(\tau,z)=(\tau,z+p\tau+q).
\)
Writing $\tau'=\tau$ and $z'=z+p\tau+q$, we have
\[
\frac{\operatorname{Im}z'_\alpha}{s}
=\frac{\operatorname{Im}z_\alpha}{s}+p_\alpha.
\]
The chain rule also gives
\[
dT_\beta\left(\frac{\partial}{\partial\tau}\right)
=
\frac{\partial}{\partial\tau'}
+\sum_{\alpha=1}^2p_\alpha
\frac{\partial}{\partial z'_\alpha},
\qquad
dT_\beta\left(\frac{\partial}{\partial z_\alpha}\right)
=
\frac{\partial}{\partial z'_\alpha}.
\]
Consequently,
\[
dT_\beta(V)
=
\frac{\partial}{\partial\tau'}
+\sum_{\alpha=1}^2
\frac{\operatorname{Im}z'_\alpha}{s}
\frac{\partial}{\partial z'_\alpha},
\]
which is the same formula at the translated point.
Thus $V$ descends to the family over $\mathbb H$.

Taking $\bar\partial$ and restricting to the fiber gives
\begin{equation}
 \kappa_{\partial/\partial\tau}
 =\left.\bar\partial V\right|_{\mathcal A_\tau}
 =-\frac{1}{2\sqrt{-1}s}
   \sum_{\alpha=1}^2
   d\overline z_\alpha\otimes\frac{\partial}{\partial z_\alpha}.
 \label{eq:QM-harmonic-KS-tensor}
\end{equation}
The restriction of $\omega_{\mathrm B}$ to $\mathcal A_\tau$ is a
translation-invariant flat K\"ahler metric. The tensor in
\eqref{eq:QM-harmonic-KS-tensor} is translation invariant and hence
parallel. In particular, it is $\bar\partial$-harmonic. It is nonzero
because $s>0$, so its Dolbeault cohomology class is nonzero by Hodge
theory. The covering map $\mathbb H\to C$ is locally biholomorphic;
therefore $\rho_t(v)\ne0$ for every nonzero $v\in T_tC$. Since
$\dim_{\mathbb C}T_tC=1$, this proves injectivity.
\end{proof}
\begin{remark}
\label{rem:comparison-Yuan}
The injectivity of the Kodaira--Spencer map can also be
deduced from \cite[Theorem~4.2]{Yuan2024}.
Work on the pullback family over $\mathbb H$, and write
\(
\sigma(\mu)=
\begin{pmatrix}
a&b\\
c&d
\end{pmatrix}.
\)
For each $\tau\in\mathbb H$, consider the cotangent
Kodaira--Spencer map evaluated at $\partial/\partial\tau$,
\(
H^0(A_\tau,\Omega^1_{A_\tau})
\to H^1(A_\tau,\mathcal O_{A_\tau}).
\)
The principal polarization identifies the target with
$\operatorname{Lie}(A_\tau)$. Denote the resulting map by
$\Phi_\tau$. In Yuan's convention, its matrix in the frames
$(dz_1,dz_2)$ and
$(\partial/\partial z_1,\partial/\partial z_2)$ is
\[
\Phi_\tau
=\frac{1}{2\pi\mathrm{i}}
\begin{pmatrix}
b&-a\\
d&-c
\end{pmatrix}.
\]
Since $ad-bc=q_D(\mu)=d_D$, we obtain
\(
\det\Phi_\tau
=\frac{d_D}{(2\pi\mathrm{i})^2}\ne0.
\)

If the ordinary Kodaira--Spencer class
$\rho_\tau(\partial/\partial\tau)$ were zero, the
cotangent Kodaira--Spencer map in this direction would
also vanish, contradicting the invertibility of
$\Phi_\tau$. Thus this class is nonzero at every point.
Since the base is one-dimensional and
$\mathbb H\to C$ is locally biholomorphic, the
Kodaira--Spencer map over $C$ is everywhere injective.
This gives an independent verification of the conclusion
obtained from the harmonic representative in
\eqref{eq:QM-harmonic-KS-tensor}.
\end{remark}

\subsection{The compact examples and their curvature}
\label{subsec:compact-product-final}

\begin{proof}[Proof of Theorem~\ref{thm:compact-Shimura-MA-example}]
Take the family \eqref{eq:universal-QM-surface} over a sufficiently
small-level compact Shimura curve. Its total space is a smooth projective
threefold. Proposition~\ref{prop:Betti-form-MA} supplies the required
Monge--Amp\`ere form, and
Proposition~\ref{prop:Shimura-KS-injective} proves injectivity of the
Kodaira--Spencer map at every point. In particular the family is
non-isotrivial.
\end{proof}

\begin{proof}[Proof of Theorem~\ref{thm:compact-product-counterexample}]
Take two copies $\pi_i:(\mathcal A_i,\omega_{\mathrm B,i})\to C_i$
of this family and put
\[
 \mathcal X=\mathcal A_1\times\mathcal A_2,
 \quad B=C_1\times C_2,
 \quad p=\pi_1\times\pi_2,
 \quad
 \omega=q_1^*\omega_{\mathrm B,1}+q_2^*\omega_{\mathrm B,2}.
\]
Then $\mathcal X$ is a smooth projective sixfold and $B$ is a smooth
projective surface. The fibers are connected abelian fourfolds. By
Propositions~\ref{prop:product-stability-MA} and
\ref{prop:product-KS-injective}, the form $\omega$ is relatively
K\"ahler, satisfies $\omega^5=0$, and the Kodaira--Spencer map of $p$
is injective everywhere.

Let $V_i=\int_{\mathcal A_{i,t_i}}\omega_{\mathrm B,i}^{2}/2$.
These volumes are positive constants. By
Proposition~\ref{prop:product-WP-splitting}, the generalized
Weil--Petersson metric is
\[
 G_{\mathrm{WP}}
 =V_2\operatorname{pr}_1^*G_{\mathrm{WP},1}
  +V_1\operatorname{pr}_2^*G_{\mathrm{WP},2}.
\]
Consequently, at every $(t_1,t_2)\in B$, any nonzero
$\xi_1\in T_{t_1}C_1$ and $\eta_2\in T_{t_2}C_2$ give
\[
 R^{\mathrm{WP}}\bigl((\xi_1,0),\overline{(\xi_1,0)},
              (0,\eta_2),\overline{(0,\eta_2)}\bigr)=0.
\]
The metric nevertheless has nonpositive holomorphic bisectional
curvature and strictly negative holomorphic sectional curvature. Indeed,
Theorem~\ref{thm:WW-curvature}, with relative dimension $4$ and
fiber volume $V_1V_2$, gives
\[
 R^{\mathrm{WP}}(\zeta,\overline\zeta,\eta,\overline\eta)
 \le -\frac{1}{2V_1V_2}
       |\langle\zeta,\eta\rangle_{\mathrm{WP}}|^2.
\]
Setting $\eta=\zeta\ne0$ yields
$\HSC(G_{\mathrm{WP}})(\zeta)\le-1/(2V_1V_2)<0$.
Thus the zero bisectional curvature occurs between nonzero, orthogonal
directions whose Kodaira--Spencer classes are both nonzero.
\end{proof}

\bibliographystyle{amsalpha}
\bibliography{bisectional_fibrations}

@article{TY11,
  author = {To, Wing-Keung and Yeung, Sai-Kee},
  title = {{K\"ahler} metrics of negative holomorphic bisectional curvature on {Kodaira} surfaces},
  journal = {Bulletin of the London Mathematical Society},
  volume = {43},
  number = {3},
  year = {2011},
  pages = {507--512},
  doi = {10.1112/blms/bdq117}
}

@article{Wan26,
  author = {Wan, Xueyuan},
  title = {{K\"ahler} metrics of the negative holomorphic (bi)sectional curvature on a compact relative {K\"ahler} fibration},
  journal = {Bulletin of the London Mathematical Society},
  volume = {58},
  number = {2},
  year = {2026},
  pages = {e70291},
  doi = {10.1112/blms.70291},
  eprint = {2409.14650},
  archivePrefix = {arXiv},
  primaryClass = {math.DG}
}

@article{WanWang2023,
  author = {Wan, Xueyuan and Wang, Xu},
  title = {Curvature of the base manifold of a {Monge--Amp\`ere} fibration and its existence},
  journal = {Mathematische Annalen},
  volume = {387},
  number = {1-2},
  year = {2023},
  pages = {353--387},
  doi = {10.1007/s00208-022-02475-9},
  eprint = {1908.03955},
  archivePrefix = {arXiv},
  primaryClass = {math.AG}
}

@article{Yuan2024,
  author = {Yuan, Xinyi},
  title = {Explicit {Kodaira--Spencer} maps over {Shimura} curves},
  journal = {Acta Mathematica Sinica, Chinese Series},
  volume = {67},
  number = {2},
  year = {2024},
  pages = {227--249},
  doi = {10.12386/A20220154},
  eprint = {2205.11334},
  archivePrefix = {arXiv},
  primaryClass = {math.NT},
  note = {English version: arXiv:2205.11334v2}
}

@article{DimitrovGaoHabegger2021,
  author = {Dimitrov, Vesselin and Gao, Ziyang and Habegger, Philipp},
  title = {Uniformity in {Mordell--Lang} for curves},
  journal = {Annals of Mathematics},
  series = {2},
  volume = {194},
  number = {1},
  year = {2021},
  pages = {237--298},
  doi = {10.4007/annals.2021.194.1.4}
}

@article{Jab09,
  author = {Jabbusch, Kelly},
  title = {Positivity of cotangent bundles},
  journal = {Michigan Mathematical Journal},
  volume = {58},
  number = {3},
  year = {2009},
  pages = {723--744},
  doi = {10.1307/mmj/1260475697},
  eprint = {0803.0622},
  archivePrefix = {arXiv},
  primaryClass = {math.AG}
}

@misc{Dem12,
  author = {Demailly, Jean-Pierre},
  title = {Complex Analytic and Differential Geometry},
  year = {2012},
  note = {Online book, version of June 21, 2012},
  url = {https://www-fourier.univ-grenoble-alpes.fr/~demailly/manuscripts/agbook.pdf}
}

@misc{Stacks,
  author = {{The Stacks Project Authors}},
  title = {The {Stacks Project}},
  year = {2026},
  howpublished = {\url{https://stacks.math.columbia.edu}},
}

@article{Par91,
  author = {Pardini, Rita},
  title = {Abelian covers of algebraic varieties},
  journal = {Journal f\"ur die reine und angewandte Mathematik},
  volume = {417},
  year = {1991},
  pages = {191--213},
  doi = {10.1515/crll.1991.417.191}
}

@article{Wolpert1986,
  author = {Wolpert, Scott A.},
  title = {Chern forms and the {Riemann} tensor for the moduli space of curves},
  journal = {Inventiones Mathematicae},
  volume = {85},
  number = {1},
  year = {1986},
  pages = {119--145},
  doi = {10.1007/BF01388794}
}

@article{Che89,
  author  = {Cheung, C.-K.},
  title   = {Hermitian metrics of negative holomorphic sectional curvature on some hyperbolic manifolds},
  journal = {Mathematische Zeitschrift},
  volume  = {201},
  year    = {1989},
  pages   = {105--119},
  doi     = {10.1007/BF01161998}
}

@article{Tsa89,
  author  = {Tsai, I.-H.},
  title   = {Negatively curved metrics on {Kodaira} surfaces},
  journal = {Mathematische Annalen},
  volume  = {285},
  year    = {1989},
  pages   = {369--379},
  doi     = {10.1007/BF01455062}
}

@article {MR436056,
    AUTHOR = {Yang, Paul C.},
     TITLE = {K\"ahler metrics on fibered manifolds},
   JOURNAL = {Proc. Amer. Math. Soc.},
  FJOURNAL = {Proceedings of the American Mathematical Society},
    VOLUME = {63},
      YEAR = {1977},
    NUMBER = {1},
     PAGES = {131--133},
      ISSN = {0002-9939,1088-6826},
   MRCLASS = {53C55 (32J15)},
  MRNUMBER = {436056},
MRREVIEWER = {A.\ L.\ Vitter, III},
       DOI = {10.2307/2041081},
       URL = {https://doi-org.vnlib.kias.re.kr/10.2307/2041081},
}

@article {Wells,
    AUTHOR = {Wells, Jr., R. O.},
     TITLE = {Comparison of de {R}ham and {D}olbeault cohomology for proper
              surjective mappings},
   JOURNAL = {Pacific J. Math.},
  FJOURNAL = {Pacific Journal of Mathematics},
    VOLUME = {53},
      YEAR = {1974},
     PAGES = {281--300},
      ISSN = {0030-8730,1945-5844},
   MRCLASS = {32J25},
  MRNUMBER = {367307},
MRREVIEWER = {Daniel\ M.\ Burns, Jr.},
       URL = {http://projecteuclid-org.vnlib.kias.re.kr/euclid.pjm/1102911803},
}

@book {MR1288523,
    AUTHOR = {Griffiths, Phillip and Harris, Joseph},
     TITLE = {Principles of algebraic geometry},
    SERIES = {Wiley Classics Library},
      NOTE = {Reprint of the 1978 original},
 PUBLISHER = {John Wiley \& Sons, Inc., New York},
      YEAR = {1994},
     PAGES = {xiv+813},
      ISBN = {0-471-05059-8},
   MRCLASS = {14-01},
  MRNUMBER = {1288523},
       DOI = {10.1002/9781118032527},
       URL = {https://doi-org.vnlib.kias.re.kr/10.1002/9781118032527},
}

@incollection{DonagiWitten2015,
  author    = {Donagi, Ron and Witten, Edward},
  title     = {Supermoduli space is not projected},
  booktitle = {String-Math 2012},
  series    = {Proceedings of Symposia in Pure Mathematics},
  volume    = {90},
  pages     = {19--71},
  publisher = {American Mathematical Society},
  address   = {Providence, RI},
  year      = {2015},
  doi       = {10.1090/pspum/090/01525},
  url       = {https://arxiv.org/abs/1304.7798}
}

@book{BertinRomagny2011,
  author    = {Bertin, Jos{\'e} and Romagny, Matthieu},
  title     = {Champs de {Hurwitz}},
  series    = {M{\'e}moires de la Soci{\'e}t{\'e}
               Math{\'e}matique de France, Nouvelle s{\'e}rie},
  number    = {125--126},
  publisher = {Soci{\'e}t{\'e} Math{\'e}matique de France},
  year      = {2011},
  doi       = {10.24033/msmf.437},
  url       = {https://www.numdam.org/item/MSMF_2011_2_125-126__1_0/}
}

@article{LemosTorzewski2023,
  author  = {Lemos, Pedro and Torzewski, Alex},
  title   = {Bounds on the number of rational points
             of curves in families},
  journal = {Bulletin of the London Mathematical Society},
  volume  = {55},
  number  = {2},
  pages   = {1019--1032},
  year    = {2023},
  doi     = {10.1112/blms.12774}
}

@article{BedoyaGoncalves2010,
  author  = {Bedoya, N. A. V. and Gon{\c{c}}alves, D. L.},
  title   = {Decomposability problem on branched coverings},
  journal = {Sbornik: Mathematics},
  volume  = {201},
  number  = {12},
  year    = {2010},
  pages   = {1715--1730},
}

@book{EsnaultViehweg1992,
  author    = {Esnault, H{\'e}l{\`e}ne and Viehweg, Eckart},
  title     = {Lectures on Vanishing Theorems},
  series    = {DMV Seminar},
  volume    = {20},
  publisher = {Birkh{\"a}user},
  address   = {Basel},
  year      = {1992},
  url       = {https://page.mi.fu-berlin.de/esnault/books/esvibuch.pdf}
}

@book{Voight2021,
  author    = {Voight, John},
  title     = {Quaternion Algebras},
  series    = {Graduate Texts in Mathematics},
  volume    = {288},
  publisher = {Springer},
  address   = {Cham},
  year      = {2021},
  doi       = {10.1007/978-3-030-56694-4},
  url       = {https://doi.org/10.1007/978-3-030-56694-4}
}

@phdthesis{Lan2008,
  author = {Lan, Kai-Wen},
  title  = {Arithmetic Compactifications of {PEL}-Type
            {Shimura} Varieties},
  school = {Harvard University},
  year   = {2008},
  url    = {https://www.kwlan.org/articles/cpt-PEL-type-thesis-single.pdf}
}

@article{Buzzard1997,
  author  = {Buzzard, Kevin},
  title   = {Integral models of certain {Shimura} curves},
  journal = {Duke Mathematical Journal},
  volume  = {87},
  number  = {3},
  year    = {1997},
  pages   = {591--612},
  doi     = {10.1215/S0012-7094-97-08719-6}
}
\end{document}